\documentclass[article]{elsarticle}
\usepackage{amsmath,amssymb,amsthm,enumerate,amsfonts,setspace}
\usepackage{geometry}
\usepackage {multirow}
\usepackage{tabularx}
\usepackage{makecell}
\usepackage{graphicx, subfigure}
\usepackage[ruled]{algorithm2e}
\usepackage{float} 

\usepackage{xcolor}
\usepackage[normalem]{ulem} 
\newcommand\hl{\bgroup\markoverwith
	{\textcolor[rgb]{0.8,0.8,0.8}{\rule[-.5ex]{1pt}{2.5ex}}}\ULon}
\usepackage[font={bf},textfont=md,labelfont={footnotesize},labelformat={default}]{caption}
\usepackage{caption}
\usepackage[colorlinks,bookmarksopen,bookmarksnumbered,citecolor=blue, linkcolor=blue, urlcolor=blue]{hyperref}

\newtheorem{theo}{Theorem}
\newtheorem{defi}{Definition}
\newtheorem{lem}{Lemma}
\newtheorem{rem}{Remark}
\newtheorem{exam}{Example}

\newtheorem{prop}{Proposition}
\newtheorem{prob}{Problem}

\begin{document}

\begin{frontmatter}

\title{A Barzilai-Borwein Descent Method for Multiobjective Optimization Problems}



\author[mymainaddress]{Jian Chen}
\author[mysecondaryaddress]{Liping Tang}
\author[mysecondaryaddress,mythirdaryaddress]{Xinmin Yang\corref{mycorrespondingauthor}}
\cortext[mycorrespondingauthor]{Corresponding author.\\ \indent \indent  Email addresses: \href{mailto:chenjian_math@163.com}{chenjian\_math@163.com} (Jian Chen), \href{mailto:tanglipings@163.com}{tanglipings@163.com} (Liping Tang), \href{mailto:xmyang@cqnu.edu.cn}{xmyang@cqnu.edu.cn} (Xinmin Yang)}

\address[mymainaddress]{Department of Mathematics, Shanghai University,        Shanghai 200444, China}
\address[mysecondaryaddress]{National Center for Applied Mathematics in Chongqing, Chongqing 401331, China}
\address[mythirdaryaddress]{School of Mathematical Sciences, Chongqing Normal University, Chongqing 401331, China}
\begin{abstract}
The steepest descent method proposed by Fliege et al. motivates the research on descent methods for multiobjective optimization, which has received increasing attention in recent years. However, empirical results show that the Armijo line search often gives a very small stepsize along the steepest direction, which decelerates the convergence seriously. This paper points out the issue is mainly due to the imbalances among objective functions. To address this issue, we propose a Barzilai-Borwein descent method for multiobjective optimization (BBDMO) that dynamically tunes gradient magnitudes using Barzilai-Borwein's rule in direction-finding subproblem. With monotone and nonmonotone line search techniques, it is proved that accumulation points generated by BBDMO are Pareto critical points, respectively. Furthermore, theoretical results indicate the Armijo line search can achieve a better stepsize in BBDMO. Finally, comparative results of numerical experiments are reported to illustrate the efficiency of BBDMO and verify the theoretical results.
\end{abstract}


\end{frontmatter}


\section{Introduction}
An unconstrained multiobjective optimization problem can be stated as follows:
\begin{align*}
	\min\limits_{x\in\mathbb{R}^{n}} F(x), \tag{MOP}\label{MOP}
\end{align*}
where $F:\mathbb{R}^{n}\rightarrow\mathbb{R}^{m}$ is a continuously differentiable function. In multiobjective optimization, there may not be a solution that simultaneously reaches the optima for all objectives, the concept of optimality is then replaced by \textit{Pareto optimality} or \textit{efficiency}. For a Pareto optimal solution, none of the objectives can be improved without sacrificing the others. Applications of this type of problems can be seen in engineering \citep{MA2004}, economics \citep{TC2007,FW2014}, management science \citep{E1984}, environmental analysis \citep{LW1992}, machine learning \citep{SK2018,YL2021}, etc. 
\par Several scalarization approaches \citep{M2012} have been devised to solve MOPs, which convert a MOP into a single-objective optimization problem (SOP), so that standard mathematical programming methods can be applied. However, the approaches burden the decision-maker to choose the parameters unknown in advance. To overcome the drawback, \citet{M1980} proposed the first descent method for MOPs, and no prior information is needed. \citet{FS2000} later independently reinvented the parameter-free method, called the steepest descent method for multiobjective optimization (SDMO). Motivated by the work of Fliege and Svaiter, some standard mathematical programming methods are extended to solve MOPs \citep[see e.g.][]{DI2004,BI2005,FD2009,QG2011,P2014,FV2016,CL2016,MP2018,LP2018,MP2019}. Unlike the steepest descent method for SOPs, the corresponding descent direction for MOPs is given by solving a direction-finding subproblem. Then, along with the produced direction, the Armijo line search guarantees sufficient descent for every objective function. Due to the monotonicity of the Armijo line search, several inequalities satisfy the line search condition simultaneously. Some literature \citep[see e.g.][]{MF2019,GK2021,MP2016} pointed out that this issue leads to a relatively small stepsize, which decelerates the convergence of SDMO.
\par In recent years, to reduce this weakness, two well-known nonmonotone line search techniques, the max-type proposed by \citet{GL1986} and the average-type proposed by \citet{ZH2004}, are suggested to apply in MOPs. For example, \citet{QJ2017} studied the max-type nonmonotone line search for vector optimization and employed the new method to portfolio management. Instead, \citet{FS2019} combined the classical projected gradient method for multiobjective optimization with average-type line search technique. Recently, \citet{ZY2022} proved the linear convergence of average-type nonmonotone projected gradient method for multiobjective optimization. Parallel to these works, \citet{MF2019} proposed two nonmonotone line search techniques and a hybrid-type line search for multiobjective problems. \citet{GK2021} applied two nonmonotone line search techniques in the quasi-Newton method and proposed an adaptive nonmonotone line search scheme. Besides, \citet{MP2016} extended Barzilai-Borwein's rule to solve MOPs. 
\par Based on the above review, instead of modifying the steepest descent direction, all related research focused on addressing the issue by using stepsize strategies. In fact, along with some well-designed descent directions such as Newton direction \citep{FD2009} and quasi-Newton direction \citep{P2014}, the monotone line search can accept unit stepsize in some special cases. Naturally, questions arise as to why monotonic line search leads to a small stepsize in SDMO and how to handle it effectively. Firstly, we point out that many inequalities in monotonic line search are inadequate to explain the first question. An example shows that monotonic line search could give a small stepsize in bi-objective optimization with strongly convex objective functions. According to the example, we conclude that the small stepsize is mainly due to the imbalances among objective functions. To address this issue, in this paper, we propose a Barzilai-Borwein descent method for multiobjective optimization (BBDMO) based on Barzilai-Borwein's rule \citep{BB1988}, which dynamically tunes gradient magnitudes in direction-finding subproblem. Indeed, the proposed method generates a sequence of new descent directions, called Barzilai-Borwein descent directions. Along with the Barzilai-Borwein descent direction, the Armijo line search can achieve a better stepsize when compared with SDMO. It provides an answer to the second question.
\par The organization of the paper is as follows. Some notations and definitions are given in Sect. 2 for our later use. Sect. 3 concludes the reasons for small stepsize in SDMO. Sect. 4 is devoted to introducing BBDMO and proving the convergence of BBDMO with different line search techniques. Numerical results are presented in Sect. 5, which demonstrates that BBDMO can handle the imbalances among objective functions. At the end of the paper, some conclusions are drawn.
\section{Notations and Definitions}
We give some notations used in this paper.
\begin{itemize}
	\item[$\bullet$] $[m]=\{1,2,...,m\}$.
	\item[$\bullet$]$\Delta_{m}=\left\{\lambda:\sum\limits_{i\in[m]}\lambda_{i}=1,\lambda_{i}\geq0,\ i\in[m]\right\}$ the $m$-dimensional unit simplex. 
	\item[$\bullet$] $\mathbb{R}_{+}$ the set of nonnegative real numbers, $\mathbb{R}_{++}$ the set of strictly positive real numbers.
	\item[$\bullet$] $\|\cdot\|$ the Euclidean distance in $\mathbb{R}^{n}$.
	\item[$\bullet$] $JF(x)\in\mathbb{R}^{m\times n}$, $\nabla F_{i}(x)\in\mathbb{R}^{n}$ and $\nabla^{2}F_{i}(x)\in\mathbb{R}^{n\times n}$ the Jacobian matrix, the gradient and the Hessian matrix of $F_{i}$ at $x$, respectively.
	\item[$\bullet$] $d(a,B) =\inf\limits_{b\in B}\{\|a-b\|\}$ the distance between the point $a$ and the set $B$.
\end{itemize}
To optimize $F$, we present the definition of optimal solutions in the Pareto sense. We introduce partial order induced by $\mathbb{R}^{m}_{+}=\mathbb{R}_{+}\times\cdots\times\mathbb{R}_{+}$:
$$F(y)\leqslant F(z)\ \Leftrightarrow\ F(z)-F(y)\in\mathbb{R}^{m}_{+}.$$
Some definitions used in this paper are given below.
\begin{defi}\rm\citep{M2012}
	A vector $x^{\ast}\in\mathbb{R}^{n}$ is called Pareto optimum to (\ref{MOP}), if there exists no $x\in\mathbb{R}^{n}$ such that $F(x)\leqslant F(x^{\ast})$ and $F(x)\neq F(x^{\ast})$.
\end{defi}

\begin{defi}\label{d1}\rm\citep{FS2000}
	A vector $x^{\ast}\in\mathbb{R}^{n}$ is called  Pareto critical point to (\ref{MOP}), if
	$$\mathrm{range}(JF(x^{*}))\cap-\mathbb{R}_{++}^{m}=\emptyset,$$
	where $\mathrm{range}(JF(x^{*}))$ denotes the range of linear mapping given by the matrix $JF(x^{*})$.
\end{defi}

\begin{defi}\rm\citep{FS2000}
	A vector $d\in\mathbb{R}^{n}$ is called descent direction for $F$ at $x$, if
	$$JF(x)d\in-\mathbb{R}_{++}^{m}.$$
\end{defi}
\section{ Steepest Descent Method for MOPs}
\par For $x\in\mathbb{R}^{n}$, recall that $d(x)$, the steepest descent direction, is defined as the optimal solution of 
\begin{align}\label{eq3.1}
	\min\limits_{d\in\mathbb{R}^{n}} \max\limits_{i\in[m]}\ \langle \nabla F_{i}(x),d\rangle+\frac{1}{2}\|d\|^{2}.
\end{align}
Since $\nabla F_{i}(x)^{T}d+\frac{1}{2}\|d\|^{2}$ is strongly convex for $i\in[m]$, then (\ref{eq3.1}) has a unique minimizer. We denote by $d(x)$ and $\theta(x)$ the optimal solution and optimal value of (\ref{eq3.1}), respectively. Hence,
\begin{equation}\label{E2.1}
	\theta(x) = \min\limits_{d\in \mathbb{R}^{n}}\max\limits_{i\in[m]}\langle \nabla F_{i}(x),d\rangle+\frac{1}{2}\|d\|^{2},
\end{equation}
and
\begin{equation}\label{E2.2}
	d(x) = \mathop{\arg\min}\limits_{d\in \mathbb{R}^{n}}\max\limits_{i\in[m]}\langle \nabla F_{i}(x),d\rangle+\frac{1}{2}\|d\|^{2}.
\end{equation}
Indeed, problem (\ref{eq3.1}) can be rewritten equivalently as the following smooth quadratic problem:
\begin{align*}\tag{QP}\label{QP}
	&\min\limits_{(t,d)\in\mathbb{R}\times\mathbb{R}^{n}}\ t+\frac{1}{2}\|d\|^{2},\\
	&\ \ \ \ \mathrm{ s.t.} \ \langle \nabla F_{i}(x),d\rangle \leq t,\ i\in[m].
\end{align*}
Notice that (\ref{QP}) is convex with linear constraints, then strong duality holds. The Lagrangian of (\ref{QP}) is
$$L((t,d),\lambda)=t+\frac{1}{2}\|d\|^{2} +\sum\limits_{i\in[m]}\lambda_{i}(\langle \nabla F_{i}(x),d\rangle-t).$$
By Karush-Kuhn-Tucker (KKT) conditions, we have
\begin{equation}\label{E3.3}
	\sum\limits_{i\in[m]}\lambda_{i}=1,
\end{equation}

\begin{equation}\label{E3.4}
	d + \sum\limits_{i\in[m]}\lambda_{i}\nabla F_{i}(x)=0,
\end{equation}

\begin{equation}\label{E3.5}
	\langle \nabla F_{i}(x),d\rangle \leq t,\ i\in[m],
\end{equation}

\begin{equation}\label{E3.6}
	\lambda_{i}\geq0,\ i\in[m],
\end{equation}

\begin{equation}\label{E3.7}
	\lambda_{i}(\langle \nabla F_{i}(x),d\rangle-t)=0,\ i\in[m].
\end{equation}
From (\ref{E3.4}), we obtain
\begin{equation}\label{E3.8}
	d(x) = -\sum\limits_{i\in[m]}\lambda_{i}(x)\nabla F_{i}(x),
\end{equation}
where $\lambda(x)=(\lambda_{1}(x),\lambda_{2}(x),...,\lambda_{m}(x))$ is the solution of dual problem:
\begin{align*}\tag{DP}\label{DP}
	-&\min\limits_{\lambda}\frac{1}{2} \|\sum\limits_{i\in[m]}\lambda_{i}\nabla F_{i}(x)\|^{2}\\
	&\mathrm{ s.t.} \ \lambda\in\Delta_{m}.	
\end{align*}
Recall that strong duality holds, we obtain
\begin{equation}\label{E3.9}
	\theta(x)=-\frac{1}{2} \|\sum\limits_{i\in[m]}\lambda_{i}(x)\nabla F_{i}(x)\|^{2}=-\frac{1}{2}\|d(x)\|^{2}.
\end{equation}
From (\ref{E3.5}), we have
\begin{equation}\label{E3.10}
	\langle \nabla F_{i}(x),d(x)\rangle \leq t(x)=\theta(x)-\frac{1}{2}\|d(x)\|^{2}=-\|d(x)\|^{2},\ i\in[m].
\end{equation}
If $\lambda_{i}(x)\neq0$, then (\ref{E3.7}) leads to
\begin{equation}\label{e3.10}
	\langle \nabla F_{i}(x),d(x)\rangle = t(x)=-\|d(x)\|^{2}.
\end{equation}
Denote $\mathcal{A}(x):=\{i:\nabla F_{i}(x)^{T}d(x)=t(x)\}$ is the set of active constraints at $x$.
It follows by (\ref{E3.3}) and (\ref{E3.7}) that
\begin{equation}\label{E3.11}
	\sum\limits_{i\in[m]}\lambda_{i}(x)\nabla F_{i}(x)= t(x)=-\|d(x)\|^{2}.
\end{equation}
Next, we will present some properties of $\theta(x)$ and $d(x)$.
\begin{lem}\rm\label{l3.1}\citep[Lemma 1]{FS2000}
	For the predefined $\theta(x)$ and $d(x)$, we have
	\begin{itemize}
		\item[$\mathrm{(a)}$] the following conditions are equivalent:
		\subitem$\mathrm{(i)}$ The point $x$ is Pareto critical;
		\subitem$\mathrm{(ii)}$ $\theta(x)=0$;
		\subitem$\mathrm{(iii)}$ $d(x)=0.$
		\item[$\mathrm{(b)}$] the following conditions are equivalent:
		\subitem$\mathrm{(i)}$ The point $x$ is non-critical;
		\subitem$\mathrm{(ii)}$ $\theta(x)<0$;
		\subitem$\mathrm{(iii)}$ $d(x)\neq0.$
		\item[$\mathrm{(c)}$] the mappings $x\rightarrow d(x)$ and $x\rightarrow\theta(x)$ are continuous.
	\end{itemize}
\end{lem}
\par For every iteration $k$, after obtaining the unique steepest descent direction $d^{k}\neq0$, the classical Armijo technique is applied for line search. 

\begin{algorithm}  
	\caption{\ttfamily Armijo\_line\_search} 
	\KwData{$x^{k}\in\mathbb{R}^{n},d^{k}\in\mathbb{R}^{n},JF(x^{k})\in\mathbb{R}^{m\times n},\sigma,\gamma\in(0,1),\beta=1$}
	\While{$F(x^{k}+\beta d^{k})- F(x^{k}) \nleq \sigma\beta JF(x^{k})d^{k}$}{~\\$\beta\leftarrow \gamma\beta$ }{$\beta^{k}\leftarrow \beta$} 
	\label{alg1} 
\end{algorithm}
\par The following result shows that the Armijo technique will accept a stepsize along with $d(x)$.
\begin{lem}\rm\label{l3.2}\citep[Lemma 4]{FS2000}
	Assume that $x\in\mathbb{R}^{n}$ is a noncritical point of $F$. Then, for any $\sigma\in(0,1)$ there exists $\beta_{0}\in (0,1]$ such that
	$$F_{i}(x+\beta d(x)) - F_{i}(x)\leq \sigma \beta\theta(x)$$
	holds for all $\beta\in [0,\beta_{0}]$ and $i\in[m]$.
\end{lem}
The stepsize obtained by Algorithm \ref{alg1} has a lower bound. 
\begin{lem}\rm\label{l3.3}\citep[Lemma 3.1]{FV2019}
	Assume $\nabla F_{i}$ is Lipschitz continuous with constant $L_{i},\ i\in[m]$, then the stepsize generated by Algorithm \ref{alg1} satisfies $\beta^{k}\geq\beta^{\min}:=\min\{\frac{2\gamma(1-\sigma)}
	{L_{\max}},1\}$, where $L_{\max}:=\max\{L_{i}:i\in[m]\}$.
\end{lem}
We also give an upper bound of the stepsize obtained by Algorithm \ref{alg1}.
\begin{lem}\rm\label{l3.4}
	Assume $F_{i}$ is strongly convex with modulus $\mu_{i},\ i\in[m]$, then the stepsize generated by Algorithm \ref{alg1} satisfies $\beta^{k}\leq\beta^{\max}:=\min\{\frac{2(1-\sigma)}
	{\mu_{\max}},1\}$, where $\mu_{\max}:=\max\{\mu_{i}:i\in\mathcal{A}(x^{k})\}$.
\end{lem}
\begin{proof} By the line search condition in Algorithm \ref{alg1}, we have
	\begin{equation}\label{e4.7}
		F_{i}(x^{k}+\beta^{k}d^{k})-F_{i}(x^{k}) \leq \sigma\beta^{k}\langle\nabla F_{i}(x^{k}),d^{k}\rangle,\ \forall i\in[m].
	\end{equation}
	On the other hand, from the $\mu_{i}$-strong convexity of $F_{i}$, we have
	\begin{equation}\label{e4.8}
		F_{i}(x^{k}+\beta^{k}d^{k})-F_{i}(x^{k})\geq \beta^{k}\langle\nabla F_{i}(x^{k}),d^{k}\rangle+\frac{\mu_{i}}{2}||\beta^{k}d^{k}
		||^{2},\ \forall i\in[m].
	\end{equation}
	It follows by (\ref{e4.7}) and (\ref{e4.8}) that
	$$\frac{\mu_{i}}{2}||\beta^{k}d^{k}
	||^{2}\leq(\sigma-1)\beta^{k}\langle\nabla F_{i}(x^{k}),d^{k}\rangle,\ \forall i\in[m].$$
	Hence,
	$$\max\limits_{i\in\mathcal{A}(x^{k})}\frac{\mu_{i}}{2}||\beta^{k}d^{k}
	||^{2}\leq(1-\sigma)\beta^{k}||d^{k}
	||^{2},$$
	due to the definition of $\mathcal{A}(x^{k})$.
	Then we conclude that
	$$\beta^{k}\leq\frac{2(1-\sigma)}{\max\limits_{i\in\mathcal{A}(x^{k})}\mu_{i}}=
	\frac{2(1-\sigma)}{\mu_{\max}}.$$
	Notice that $\beta^{k}\leq1$, thus,
	$$\beta^{k}\leq\beta^{\max}:=\min\{\frac{2(1-\sigma)}
	{\mu_{\max}},1\}.$$
\end{proof}
Let us now summarize Lemmas \ref{l3.3} and \ref{l3.4}. Assume $\nabla F_{i}$ is Lipschitz continuous with constant $L_{i}$ and $F_{i}$ is strongly convex with modulus $\mu_{i}$, $i\in[m]$. We have the following statement.
\begin{rem}\label{imbalance}
	The stepsize obtained by Algorithm \ref{alg1} satisfies $\min\{\frac{2\gamma(1-\sigma)}
	{L_{\max}},1\}\leq\beta^{k}\leq\min\{\frac{2(1-\sigma)}
	{\mu_{\max}},1\}$. Similarly, for a single-objective function $F_{i}$ the stepsize depends on $L_{i}$ and $\mu_{i}$. Hence, the sets $\{L_{i}:i\in[m]\}$ and $\{\mu_{i}:i\in[m]\}$ are more imbalanced (the imbalances mean the differences among elements in a set), the stepsize generated by Algorithm \ref{alg1} is relatively smaller. In general, as the number of objective functions grows, the imbalances of $\{L_{i}:i\in[m]\}$ and $\{\mu_{i}:i\in[m]\}$ increase. 
\end{rem}
However, the imbalances can be great even in bi-objective optimization problems. We will give an example to claim it.
\begin{exam}\label{exam1}
	Consider the following multiobjective optimization problem:
	$$\min\limits_{x\in\mathbb{R}^{2}}\ (f_1(x),f_2(x))$$
	where $f_{1}(x_{1},x_{2})=100(x_{1}-50)^{2}+100(x_{2}+50)^{2}$, $f_2(x_{1},x_{2})=\frac{1}{2}x_{1}^{2}+\frac{1}{2}x_{2}^{2}$. Then, we have 
	$$\nabla^{2} f_{1}(x_{1},x_{2})=\begin{pmatrix} 200&\ 0\\0&\ 200 \end{pmatrix},$$
	and
	$$\nabla^{2}f_{2}(x_{1},x_{2})=\begin{pmatrix} 1 &\ 0\\0 &\ 1 \end{pmatrix}.$$
	If $x^{k}=\begin{pmatrix} 1\\0 \end{pmatrix}$, then $\lambda_{1}^{k},\lambda_{2}^{k}\neq0$. It follows by (\ref{e3.10}) that
	$$\mathcal{A}(x^{k})=\{1,2\}.$$
	Then we have
	$$\mu_{\max}=200\ \mathrm{and}\ L_{\max}=200.$$
	According to Lemmas \ref{l3.3} and \ref{l3.4}, we obtain
	\begin{equation}\label{e3.16}
		\frac{\gamma(1-\sigma)}{100}\leq\beta^{k}\leq\frac{1-\sigma}{100}<0.01.
	\end{equation}
	Furthermore, $\|d^{k}\|\leq\|\nabla f_{2}(x^{k})\|=\|x^{k}\|$, then $\|\beta^{k}d^{k}\|<0.01$. Denote $X^{*}:=\{x:x=s(0,0)+(1-s)(50,-50),s\in[0,1]\}$ the Pareto set, we have $d(x^{k},X^{*})=\frac{\sqrt{2}}{2}$ and $d(x^{k+1},X^{*})\geq d(x^{k},X^{*})-\|\beta^{k}d^{k}\|>\frac{\sqrt{2}}{2}-0.01$. This implies $\frac{d(x^{k+1},X^{*})}{d(x^{k},X^{*})}>\frac{100-\sqrt{2}}{100}$, which leads to slow convergence.
\end{exam}

The steepest descent method for MOPs is described as follows.
\begin{algorithm}  
	\caption{{\ttfamily{steepest\_descent\_method\_for\_MOPs}}\ \citep{FS2000}}
	\SetAlgoLined  
	\KwData{$x^{0}\in\mathbb{R}^{n}$}
	\For{$k=0,1,...$}{  $\lambda^{k}\leftarrow \mathop{\arg\min}\limits_{\lambda\in\Delta_{m}}\frac{1}{2} \|\sum\limits_{i\in[m]}\lambda_{i}\nabla F_{i}(x^{k})\|^{2}$\\  
		$d^{k}\leftarrow-\sum\limits_{i\in[m]}\lambda_{i}^{k}\nabla F_{i}(x^{k})$\\
		\eIf{$d^{k}=0$}{~\\ {\bf{return}} Pareto critical point $x^{k}$ }{$\beta^{k}\leftarrow$ {\ttfamily Armijo\_line\_search}$(x^{k},d^{k},JF(x^{k}))$\\
			$x^{k+1}\leftarrow x^{k}+\beta^{k}d^{k}$}}  
	\label{alg2}
\end{algorithm}
\section{BBDMO: A Barzilai-Borwein descent Method for MOPs}
In this section, a Barzilai-Borwein descent method is devised to overcome the drawback of Algorithm \ref{alg1}. At first, we explain why the Barzilai-Borwein method is taken into consideration. Recall that the imbalances among objective functions will give a relatively small stepsize. The direct reason for the issue is that different objective functions have a similar amount of descent (due to (\ref{e3.10})) in each iteration. Since the local second-order information is not used in direction-finding subproblems, the descent direction with the same inner products among gradients will lead to a relatively small stepsize. Naturally, a descent direction with different inner products among gradients is required. Fortunately, it is valid for Newton direction. Recall the direction-finding subproblem for the Newton method:
\begin{equation}\label{E4.0}
	\min\limits_{d\in\mathbb{R}^{n}} \max\limits_{i\in[m]}\ \langle\nabla F_{i}(x),d\rangle+\frac{1}{2}\langle d,\nabla^{2}F_{i}(x)d\rangle,
\end{equation}
where  $F:\mathbb{R}^{n}\rightarrow\mathbb{R}^{m}$ is a twice continuously differentiable function. In light of the analysis in \citep{FD2009}, we have
$$\langle\nabla F_{i}(x),d(x)\rangle=-\frac{1}{2}\langle d(x),\nabla^{2}F_{i}(x)d(x)\rangle+t(x),\ \lambda_{i}(x)\neq0,$$
where $d(x)$ and $t(x)$ are optimal solution and optimal value of (\ref{E4.0}), respectively. In order to achieve this with a first-order method, we devise the following direction-finding subproblem at $x^{k}$:
\begin{equation}\label{E4.1}
	\min\limits_{d\in\mathbb{R}^{n}} \max\limits_{i\in[m]}\ \langle \nabla \hat{F}_{i}(x^{k}),d\rangle+\frac{1}{2}\|d\|^{2},
\end{equation}
where $$\nabla\hat{F}_{i}(x^{k})=\frac{1}{\alpha^{k}_{i}}\nabla F_{i}(x^{k}),\  \alpha^{k}_{i}\in\mathbb{R}_{++},\ i\in[m].$$
Similarly, denote $\hat{\lambda}^{k}$ the KKT multiplier vector. By KKT conditions, if $\hat\lambda^{k}_{i}\neq0$, then $\langle \nabla\hat{F}_{i}(x^{k}),d^{k}\rangle=-\|d^{k}\|^{2}$, i.e., $\langle \nabla F_{i}(x^{k}),d^{k}\rangle=-\alpha^{k}_{i}\|d^{k}\|^{2}$. Therefore, this raises the immediate question: how to choose an appropriate $\alpha^{k}\in\mathbb{R}^{m}$ to accelerate convergence. In view of (\ref{E4.1}), every maximum term can be rewritten as  
\begin{equation}\label{E4.2}
	\langle \nabla F_{i}(x^{k}),d\rangle+\frac{\alpha^{k}_{i}}{2}\|d\|^{2}.
\end{equation}
It is well-known by using $\alpha^{k}_{i}I$ as an approximation of $\nabla^{2}F_{i}(x^{k})$ so that one possible use for $\alpha^{k}$ could be given by Barzilai-Borwein Method. 
\subsection{Barzilai-Borwein descent Method for MOPs}
Firstly, a brief introduction of Barzilai-Borwein's method for single-objective optimization is given. To minimize a smooth function $f:\mathbb{R}^{n}\rightarrow\mathbb{R}$, the Barzilai-Borwein's method updates iterates as follows
$$x^{k+1}=x^{k}-\frac{1}{\alpha^{k}}\nabla f(x^{k}),$$
where $\alpha^{k}$ is the solution of the following problem
$$\min\limits_{\alpha}\|\alpha s^{k-1} -y^{k-1}\|^{2}$$
with $s^{k-1}=x^{k}-x^{k-1}$ and $y^{k-1}=\nabla f({x^{k}})-\nabla f(x^{k-1})$.
Then, simple calculations lead to
$$\alpha^{k}=\frac{\langle s^{k-1},y^{k-1}\rangle}{\langle s^{k-1},s^{k-1}\rangle}.$$
Along with this, we will present a Barzilai-Borwein descent method for multiobjective optimization problems. The BBDMO updates iterates as follows
$$x^{k+1}=x^{k}+d^{k}_{BB},$$
where $d^{k}_{BB}$ is the solution of (\ref{E4.1}) with 
$$ \alpha^{k}_{i}=\frac{\langle s^{k-1},y^{k-1}_{i}\rangle}{\langle s^{k-1},s^{k-1}\rangle},\ i\in[m],$$
and
$$s^{k-1}=x^{k}-x^{k-1},\ y^{k-1}_{i}=\nabla F_{i}({x^{k}})-\nabla F_{i}(x^{k-1}),\ i\in[m].$$
Recall the Barzilai-Borwein's method for multiobjective optimization (BBMO) problems \citep{MP2016}, which updates iterates as follows
$$x^{k+1}=x^{k}+\frac{1}{\bar{\alpha}^{k}}d^{k},$$
where $d^{k}$ is the multiobjective steepest descent direction and 
$$\bar{\alpha}^{k}=\frac{\langle s^{k-1},y^{k-1}\rangle}{\langle s^{k-1},s^{k-1}\rangle}$$ 
with
$$s^{k-1}=x^{k}-x^{k-1},\ y^{k-1}_{i}=d^{k}-d^{k-1}.$$
Next, we give the relations of BBDMO between BBMO and Barzilai-Borwein's method for single-objective optimization problems. 
\begin{rem}
	If $m=1$, then $d^{k}_{BB}=-\frac{1}{\alpha^{k}}\nabla F(x^{k})$. In this case, BBDMO reduces to the Barzilai-Borwein's method for single-objective optimization problems. In terms of BBMO, it is a special case of BBDMO with $\alpha^{k}_{i}=\alpha^{k}$ for all $i\in[m]$.
\end{rem}

There is a significant difference between BBDMO and BBMO. 
\begin{rem}
	The descent direction of BBMO is the same as the one in SDMO. However, the statement is usually not true for BBDMO. If $\lambda^{k}_{i},\lambda^{k}_{j}\neq0$ for some $i,j\in[m]$, then $\langle \nabla\hat{F}_{i}({x^{k}}),d^{k}_{BB}\rangle=\langle \nabla\hat{F}_{j}({x^{k}}),d^{k}_{BB}\rangle$. In general, $\langle \nabla F_{i}({x^{k}}),d^{k}_{BB}\rangle\neq\langle \nabla F_{j}({x^{k}}),d^{k}_{BB}\rangle$ due to $\alpha^{k}_{i}\neq\alpha^{k}_{j}$ ($\alpha^{k}_{i}$ and $\alpha^{k}_{j}$ are objective-based). Thus, BBDMO produces a sequence of new descent directions, which has different inner products among gradients.
\end{rem}
Similarly, we will present some properties of $d^{k}_{BB}$ as follows.
\begin{lem}\rm\label{l4.1}
	Assume $F_{i}$ is strictly convex and $\nabla F_{i}$ is Lipschitz continuous with constant $L_{i},\ i\in[m]$. For the predefined $d^{k}_{BB}$, we have 
	\begin{itemize}
		\item[$\mathrm{(a)}$] vector $d^{k}_{BB}\neq0$ is a descent direction.
		\item[$\mathrm{(b)}$] the following conditions are equivalent:
		\subitem$\mathrm{(i)}$ The point $x^{k}$ is Pareto critical;
		\subitem$\mathrm{(ii)}$ $d^{k}_{BB}=0.$
		\item[$\mathrm{(c)}$] the following conditions are equivalent:
		\subitem$\mathrm{(i)}$ the point $x^{k}$ is non-critical;
		\subitem$\mathrm{(ii)}$ $d^{k}_{BB}\neq0.$
		\item[$\mathrm{(d)}$] if there exists a convergent subsequence $x^{k}\stackrel{\mathcal{K}}{\longrightarrow} x^{*}$ such that $d^{k}_{BB}\stackrel{\mathcal{K}}{\longrightarrow}0$, then $x^{*}$ is Pareto critical.
	\end{itemize}
\end{lem}
\begin{proof}
	Since $F_{i}$ is strictly convex and $\nabla F_{i}$ is Lipschitz continuous with constant $L_{i},\ i\in[m]$, we obtain the following bounds:
	$$0<\langle \nabla F_{i}({x^{k}})-\nabla F_{i}(x^{k-1}),x^{k}-x^{k-1}\rangle\leq L_{i}\|x^{k}-x^{k-1}\|^{2},\ i\in[m].$$
	Hence
	\begin{equation}
		0<\alpha_{i}^{k}\leq L_{\max}:=\max\limits_{i\in[m]}L_{i},\ i\in[m].
	\end{equation}
	Then assertions (a)-(c) can be obtained by using the same arguments as in the proof of \citep[Lemma 1]{FS2000}. Next, we prove assertion (d). From equality (\ref{E3.8}), there exists $\lambda^{k}\in\Delta_{m}$ such that
	$d^{k}_{BB}=-\sum\limits_{i\in[m]}\lambda^{k}_{i}\frac{\nabla F_{i}(x^{k})}{\alpha_{i}^{k}}$. In what follows, we give the relation between $d^{k}_{BB}$ and the steepest descent direction $d(x^{k})$:
	\begin{align*}
		\|d^{k}_{BB}\|&=\left\|\sum\limits_{i\in[m]}\lambda^{k}_{i}\frac{\nabla F_{i}(x^{k})}{\alpha_{i}^{k}}\right\|\\
		&=(\sum\limits_{i\in[m]}\frac{\lambda^{k}_{i}}{\alpha^{k}_{i}})\left\|\sum\limits_{i\in[m]}\bar{\lambda}^{k}_{i}\nabla F_{i}(x^{k})\right\|\\
		&\geq\frac{1}{L_{\max}}\left\|\sum\limits_{i\in[m]}\bar{\lambda}^{k}_{i}\nabla F_{i}(x^{k})\right\|\\
		&\geq\frac{1}{L_{\max}}\|d(x^{k})\|,
	\end{align*}
	where the first inequality is given by $\lambda\in\Delta_{m}$ and $\alpha^{k}_{i}\leq L_{\max}$, the second inequality follows from $\bar{\lambda}\in\Delta_{m}$ and $d(x^{k})$ is the solution of (\ref{DP}). Therefore,
	$d^{k}_{BB}\stackrel{\mathcal{K}}{\longrightarrow}0$ implies $d(x^{k})\stackrel{\mathcal{K}}{\longrightarrow}0$, this together with the continuity of $d$ implies $d(x^{*})=0$. Then $x^{*}$ is Pareto critical.
\end{proof}
\begin{rem}
	Comparing with Lemmas \ref{l3.1} and \ref{l4.1}, if the mapping $x\rightarrow d_{BB}(x)$ is continuous, which implies the assertion (d) in Lemma \ref{l4.1} directly. However, the continuity of $d_{BB}$ may be invalid due to variable $\alpha\in\mathbb R^{m}_{++}$. It is worth noting that assertion (d) in Lemma \ref{l4.1} is crucial for convergence analysis.
\end{rem}
According to Lemma \ref{l4.1}, the question remains of how a positive $\alpha^{k}_{i}$ may be obtained without the strict convexity of $F_{i}$. To guarantte $d^{k}_{BB}$ is a descent direction, we address it in two cases. If $\langle s^{k-1},y^{k-1}_{i}\rangle<0$, we apply the method in \citep{DA2015}, and set 
$$\alpha_{i}^{k}=\frac{\|y^{k-1}_{i}\|}{\|s^{k-1}\|}.$$
In the case  $\langle s^{k-1},y^{k-1}_{i}\rangle=0$, such as $F_{i}$ is linear, we set
\begin{equation}\label{epsi}
	\alpha^{k}_{i}=\alpha_{\min},
\end{equation}	
where $\alpha_{\min}$ is small positive constant. With this setting, (\ref{E4.1}) can be recast as follows:
\begin{align*}
	&\min\limits_{(t,d)\in\mathbb{R}\times\mathbb{R}^{n}}\ t+\frac{1}{2}\|d\|^{2},\\
	&\ \ \ \ \mathrm{ s.t.} \ \langle \nabla F_{i}(x),d\rangle \leq \alpha^{k}_{i}t,\ i\in[m].
\end{align*}
Hence, for linear objective $F_{i}$, as $\alpha_{\min}$ tends to $0$ the associated constraint can be rewritten as
$$\langle \nabla F_{i}(x),d\rangle \leq0.$$
In what follows, we will explain the transformation. At first, we consider the following quadratic problem:
\begin{align*}\tag{QP\_k}\label{QP(k)}
	&\min\limits_{(t,d)\in\mathbb{R}\times\mathbb{R}^{n}}\ t+\frac{1}{2}\|d\|^{2},\\
	&\ \ \ \ \mathrm{ s.t.} \ \langle \nabla F_{i}(x),d\rangle \leq t,\ i\in[m-1],\\
	&\ \ \ \ \ \ \ \ \ \langle k\nabla F_{m}(x),d\rangle \leq t.
\end{align*}
Next, we discuss the relation between $k$ and the solution of (\ref{QP(k)}).
\begin{prop}\label{p4.1}
	Let $(t_{k_{1}},d_{k_{1}})$ and $(t_{k_{2}},d_{k_{2}})$ be the solution of $\mathrm{(\ref{QP(k)})}$ with $k=k_{1},k_{2}$, respectively. If $k_{2}\geq k_{1}>0$, then $\|d_{k_1}\|\leq\|d_{k_{2}}\|$ and $t_{k_1}\geq t_{k_{2}}$.
\end{prop}
\begin{proof}
	In view of equality (\ref{E3.8}), there exist $\lambda^{1},\lambda^{2}\in\Delta_{m}$ such that
	$$d_{k_{1}}=-\sum^{m-1}_{i=1}\lambda_{i}^{1}\nabla F_{i}(x)-\lambda^{1}_{m}k_{1}\nabla F_{m}(x),$$
	and
	$$d_{k_{2}}=-\sum^{m-1}_{i=1}\lambda_{i}^{2}\nabla F_{i}(x)-\lambda^{2}_{m}k_{2}\nabla F_{m}(x).$$
	Then, a direct calculation gives
	\begin{align*}
		\|d_{k_{2}}\|&=\|\sum^{m-1}_{i=1}\lambda_{i}^{2}\nabla F_{i}(x)+\lambda^{2}_{m}k_{2}\nabla F_{m}(x)\|\\
		&=\|\sum^{m-1}_{i=1}\lambda_{i}^{2}\nabla F_{i}(x)+\lambda^{2}_{m}\frac{k_{2}}{k_{1}}k_{1}\nabla F_{m}(x)\|\\
		&=(\sum^{m-1}_{i=1}\lambda_{i}^{2}+\frac{k_{2}}{k_{1}}\lambda^{2}_{m})\|\sum^{m-1}_{i=1}\frac{\lambda_{i}^{2}}{\sum^{m-1}_{i=1}\lambda_{i}^{2}+\frac{k_{2}}{k_{1}}\lambda^{2}_{m}}\nabla F_{i}(x)+\frac{\lambda_{m}^{2}}{\sum^{m-1}_{i=1}\lambda_{i}^{2}+\frac{k_{2}}{k_{1}}\lambda^{2}_{m}}k_{1}\nabla F_{m}(x)\|\\
		&\geq\|\sum^{m-1}_{i=1}\lambda_{i}^{1}\nabla F_{i}(x)+\lambda^{1}_{m}k_{1}\nabla F_{m}(x)\|\\
		&=\|d_{k_{1}}\|,
	\end{align*}
	where the inequality is given by $\sum^{m-1}_{i=1}\lambda_{i}^{2}+\frac{k_{2}}{k_{1}}\lambda^{2}_{m}\geq1$ and $\|d_{k1}\|=\arg\min\limits_{\lambda\in\Delta_{m}}\|\sum^{m-1}_{i=1}\lambda_{i}\nabla F_{i}(x)+\lambda_{m}k_{1}\nabla F_{m}(x)\|$. Consequently, $t_{k_1}\geq t_{k_{2}}$ follows by (\ref{e3.10}). The proof is complete.
\end{proof}
\begin{rem}\label{r4.3}
	If $\lambda^{2}_{m}\neq0$ and $\|d_{k_{2}}\|\neq0$, the inequalities hold strictly in Proposition \ref{p4.1}. Hence, the influence brought by the linear objective can be relieved by enlarging its gradient in direction-finding subproblems.
\end{rem}
As the above statement is abstract, we exemplify the remark as follows.
\begin{exam}
	Consider the multiobjective optimization problem:
	$$\min\limits_{x\in \mathbb{R}^{2}}(f_{1}(x),f_{2}(x),f_{3}(x))$$
	where $f_{1}(x)=5x_{1}^{2}+10x_{2}^{2}$, $f_{2}(x)=2(x_{1}-2)^{2}+5x_{2}^{2}$ and $f_{3}(x)= - x_{1}$. By simple calculations, we have
	$$\nabla f_{1}(x)=\begin{pmatrix} 10x_{1}\\20x_{2} \end{pmatrix},\ \nabla f_{2}(x)=\begin{pmatrix} 4(x_{1}-2)\\10x_{2} \end{pmatrix},\ \nabla f_{3}(x)=\begin{pmatrix} -1\\0 \end{pmatrix}.$$
	Figure 1 shows the descent directions at $x=\begin{pmatrix} 1\\-1 \end{pmatrix}$ with $k=1,5,29,+\infty$. In view of Figure 1, the descent direction obtained by $\mathrm{(\ref{QP(k)})}$ approximates to $-\nabla f_{2}(x)$ as $k$ increases. If $k\geq29$, the descent direction is the same as the one without the linear objective function, i.e., the influence brought by linear objective function can be relieved by increasing $k$ in $\mathrm{(\ref{QP(k)})}$.
	\begin{figure}[H]
		\centering
		\centering
		\subfigure[$k=1$]{
			\includegraphics[scale=0.5]{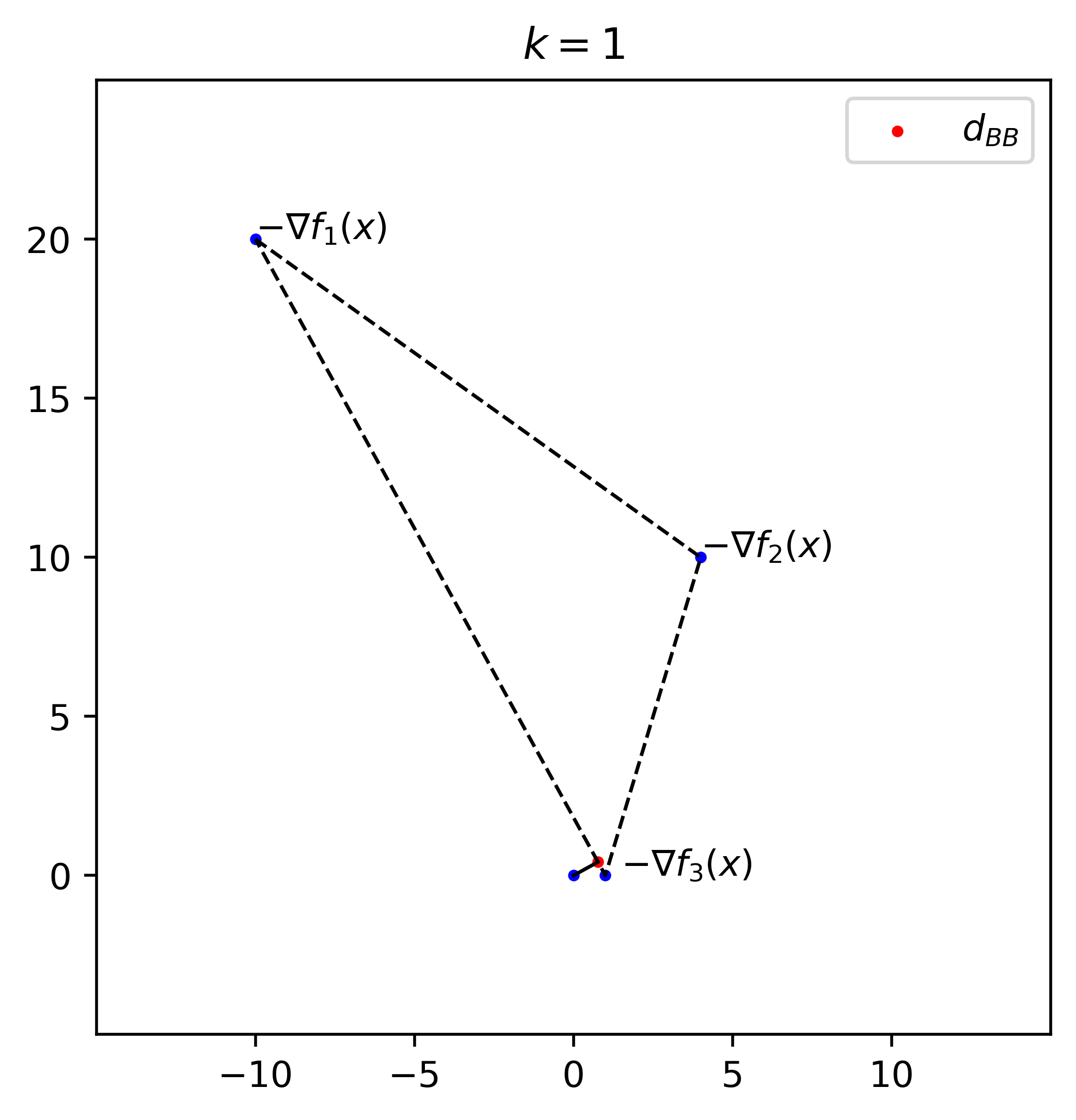}
		}
		\subfigure[$k=5$]{
			\includegraphics[scale=0.5]{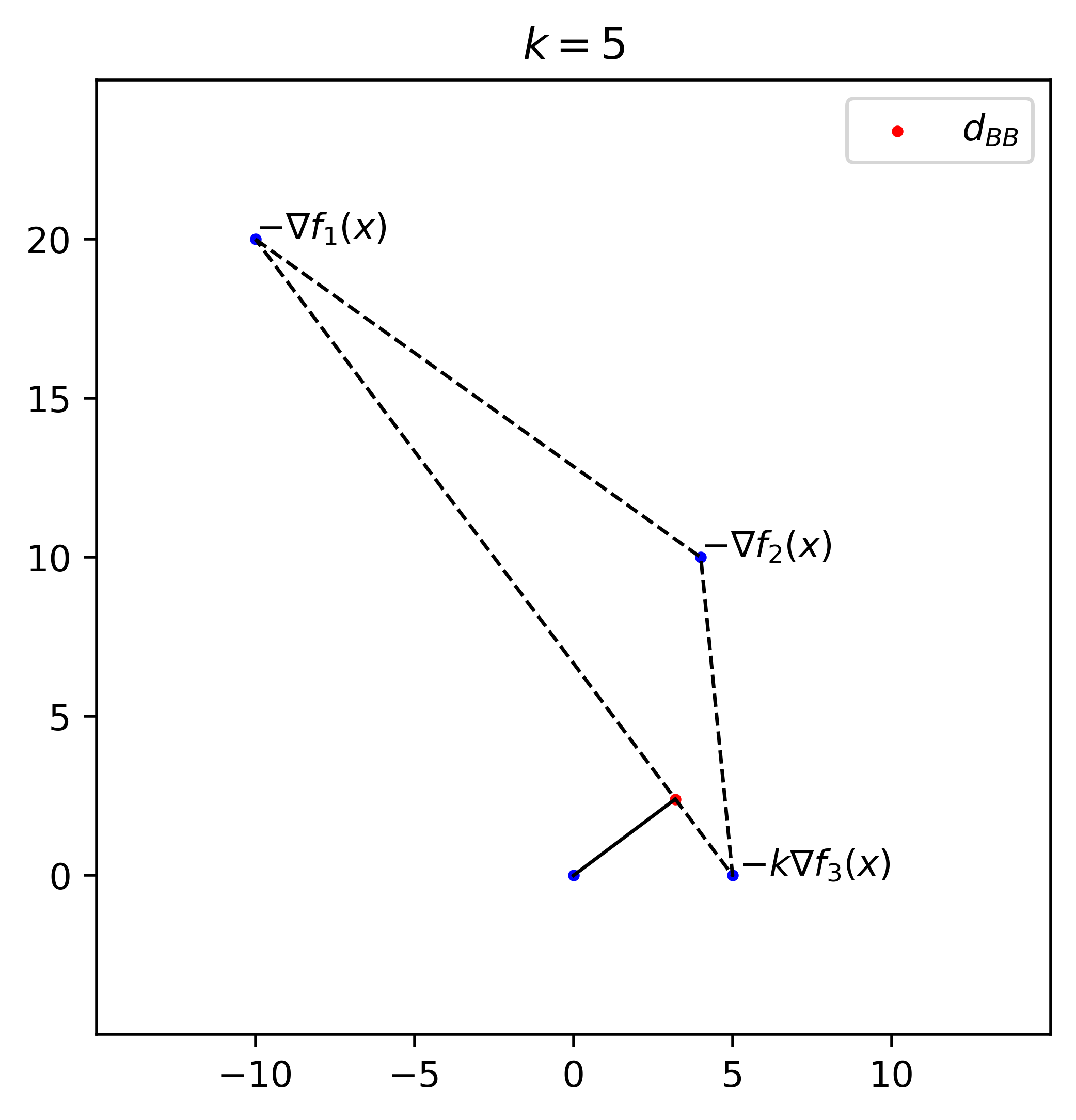} 
		}
		\subfigure[$k=29$]{
			\includegraphics[scale=0.5]{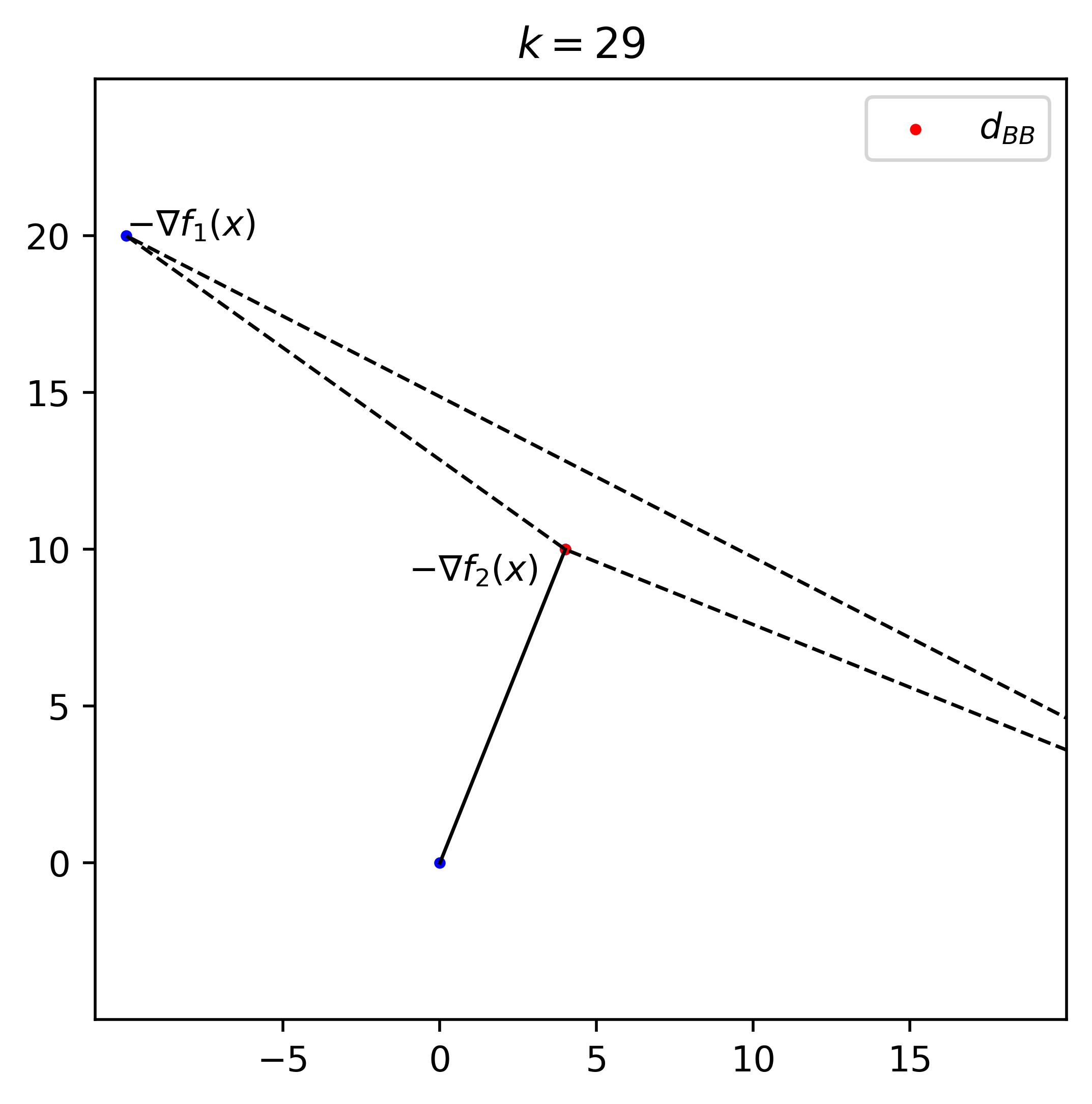}
		}
		\subfigure[$k=\infty$]{
			\includegraphics[scale=0.5]{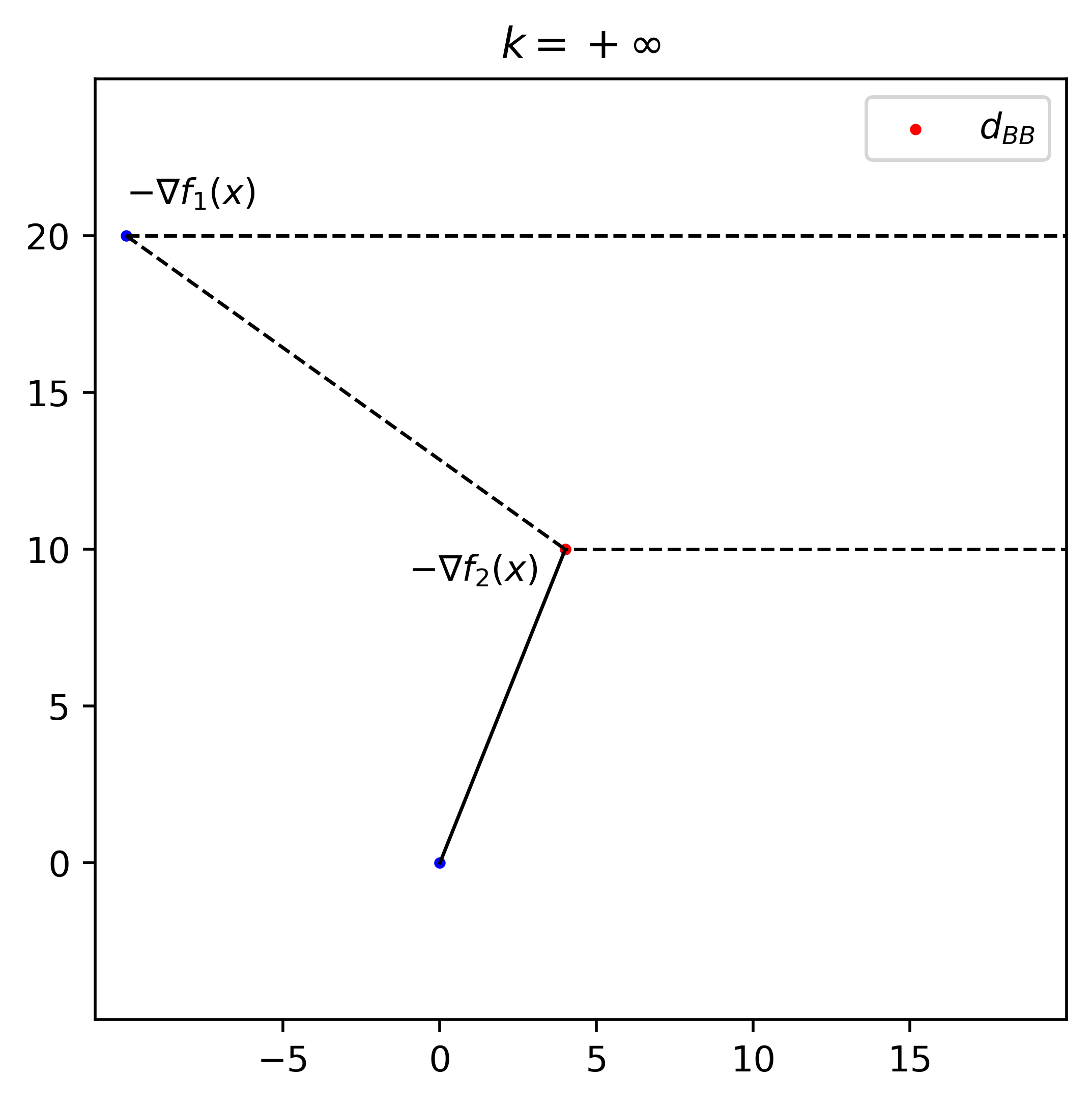} 
		}
		\caption{The descent directions obtained by solving (\ref{QP(k)}) with $k=1,5,29,\infty$.}	
	\end{figure}
\end{exam}

Recall that $\alpha^{k}$ used in SOPs always has an upper bound. In what follows, we show that the strategy is also required in MOPs. On the one hand, the upper bound of $\alpha^{k}$ is crucial to the proof of assertion (d) in Lemma \ref{l4.1}. On the other hand, since a stopping criterion is usually set based on $\|d^{k}_{BB}\|$ in practice, the large $\alpha_{i}^{k}$ may lead to a fake Pareto critical point due to $\|d^{k}_{BB}\|\leq\|\frac{\nabla F_{i}(x^{k})}{\alpha^{k}_{i}}\|$.
\par To sum up, $\alpha^{k}_{i}$ is set as follows:
\begin{equation}\label{alpha_k}
	\alpha^{k}_{i}=\left\{
	\begin{aligned}
		&\max\left\{\alpha_{\min},\min\left\{\frac{\langle s^{k-1},y^{k-1}_{i}\rangle}{\langle s^{k-1},s^{k-1}\rangle},\ \alpha_{\max}\right\}\right\},\ \langle s^{k-1},y^{k-1}_{i}\rangle>0, \\
		&\max\left\{\alpha_{\min},\min\left\{\frac{\|y^{k-1}\|}{\|s^{k-1}\|},\ \alpha_{\max}\right\}\right\},\ \ \ \ \ \ \ \ \langle s^{k-1},y^{k-1}_{i}\rangle<0,\\
		& \alpha_{\min}, \ \ \ \ \ \ \ \ \ \ \ \ \ \ \ \ \ \ \ \ \ \ \ \ \ \ \ \ \ \ \ \ \ \ \ \ \ \ \ \ \ \ \ \ \ \ \ \ \ \ \langle s^{k-1},y^{k-1}_{i}\rangle=0,
	\end{aligned}
	\right.
\end{equation}
for all $i\in[m]$, where $\alpha_{\min}$ is a small positive constant and $\alpha_{\max}$ is a large positive constant. In this setting, $0<\alpha^{k}_{i}\leq\alpha_{\max}$ for all $i\in[m]$, which implies all assertions in Lemma \ref{l4.1} are vaild.

To guarantee the global convergence, we incorporate line search techniques into BBDMO. Besides monotone line search technique such as Algorithm \ref{alg1}, we also consider the following two nonmonotone line search techniques:
\begin{algorithm}  
	\caption{\ttfamily Max-type\_nonmonotone\_line\_search} 
	\KwData{$x^{k}\in\mathbb{R}^{n},d^{k}_{BB}\in\mathbb{R}^{n},JF(x^{k})\in\mathbb{R}^{m\times n},\sigma,\gamma\in(0,1),\beta=1$, a nonnegative integer $M$.}{$C^{k}_{i}=\max\limits_{0\leq j\leq \min(k,M)}F_{i}(x^{k-j}),\ i\in[m]$}\\
	\While{$F(x^{k}+\beta d^{k}_{BB})- C^{k} \nleq \sigma\beta JF(x^{k})d^{k}_{BB}$}{~\\$\beta\leftarrow \gamma\beta$ }{$\beta^{k}\leftarrow \beta$}  
\end{algorithm}
\begin{algorithm}  
	\caption{\ttfamily Average-type\_nonmonotone\_line\_search} 
	\KwData{$x^{k}\in\mathbb{R}^{n},d^{k}_{BB}\in\mathbb{R}^{n},JF(x^{k})\in\mathbb{R}^{m\times n},\sigma,\gamma,\eta\in(0,1),\beta=1$, $q^{k-1}$, $C^{k-1}$. }{$q^{k}=\eta q^{k-1}+1$\\
		$C^{k}=\frac{\eta q^{k-1}}{q^{k}}C^{k-1}+\frac{1}{q^{k}}F(x^{k})$}\\
	\While{$F(x^{k}+\beta d^{k}_{BB})- C^{k} \nleq \sigma\beta JF(x^{k})d^{k}_{BB}$}{~\\$\beta\leftarrow \gamma\beta$ }{$\beta^{k}\leftarrow \beta$}  
\end{algorithm}

Next, we give the lower and upper bounds of stepsize for BBDMO with different line search techniques.
\begin{prop}
	Assume $\nabla F_{i}$ is Lipschitz continuous with constant $L_{i}$ and $F_{i}$ is strongly convex with modulus $\mu_{i} ,\ i\in[m]$, and let $\sigma\leq\frac{1}{2}$ in line search. Then the stepsize generated by BBDMO with either monotone or nonmonotone line search satisfies $\min\{1,\beta_{\min}\}\leq\beta^{k}\leq1$, where $\beta_{\min}:=\min\{\frac{2\gamma(1-\sigma)\mu_{i}}{L_{i}}:i\in[m]\}$.
\end{prop}
\begin{proof}
	It is sufficient to prove the assertion with monotone line search since $\beta^{k}\leq1$ for all line search techniques and stepsize generated by nonmonotone line search is greater than the one obtained by nonmonotone line search.
	In view of monotone line search, if $\beta^{k}<1$, then backtracking is conducted, so that
	\begin{equation}\label{E4.4}
		F_{i}(x^{k}+\frac{\beta^{k}}{\gamma}d^{k}_{BB})-F_{i}(x^{k})>\sigma\frac{\beta^{k}}{\gamma}\langle\nabla F_{i}(x^{k}),d^{k}_{BB}\rangle
	\end{equation}
	for some $i\in[m]$. Since $\nabla F_{i}$ is Lipschitz continuous, we have
	\begin{equation}\label{E4.5}
		F_{i}(x^{k}+\frac{\beta^{k}}{\gamma}d^{k}_{BB})-F_{i}(x^{k})\leq\frac{\beta^{k}}{\gamma}\langle\nabla F_{i}(x^{k}),d^{k}_{BB}\rangle + \frac{L_{i}}{2}\|\frac{\beta^{k}}{\gamma}d^{k}_{BB}\|^{2}
	\end{equation}
	for all $i\in[m]$. It follows by (\ref{E4.4})-(\ref{E4.5}) and $\langle\nabla F_{i}(x^{k}),d^{k}_{BB}\rangle\leq-\alpha^{k}_{i}\|d^{k}_{BB}\|^{2}$ that
	\begin{equation}\label{E4.6}
		\beta^{k}\geq\frac{2\gamma(1-\sigma)\alpha_{i}^{k}}{L_{i}}
	\end{equation}
	for some $i\in[m]$. Since $F_{i}$ is strongly convex with modulus $\mu_{i}$, we obtain the following bound:
	$$\langle \nabla F_{i}(x^{k})-\nabla F_{i}(x^{k-1}),x^{k}-x^{k-1}\rangle\geq\mu_{i}\|x^{k}-x^{k-1}\|^{2}.$$
	Thus, 
	\begin{equation}\label{E4.7}
		\alpha^{k}_{i}\geq\mu_{i}
	\end{equation}
	for all $i\in[m]$. We use (\ref{E4.6}) and (\ref{E4.7}) to get
	\begin{equation}\label{E4.8}
		\beta^{k}\geq\min\{\frac{2\gamma(1-\sigma)\mu_{i}}{L_{i}}:i\in[m]\}.
	\end{equation}
	Note that the above analysis is under the assumption $\beta^{k}<1$, then
	\begin{equation}\label{E4.9}
		\beta^{k}\geq\min\{1,\beta_{\min}\}.
	\end{equation}
	The upper bound of $\beta^{k}$ is obtained by inequality (\ref{E4.7}) and the same arguments as in the proof of Lemma \ref{l3.4}.
\end{proof}
The complete procedure of BBDMO is given as follows.
\begin{algorithm}  
	\caption{{\ttfamily{Barzilai-Borwein\_descent\_method\_for\_MOPs}}}
	\SetAlgoLined  
	\KwData{$x^{0}\in\mathbb{R}^{n}$}
	{for $k=0$, do the iteration as \ttfamily{steepest\_descent\_method\_for\_MOPs} }\\
	\For{$k=1,...$}{Update $\alpha^{k}_{i}$ as (\ref{alpha_k}),
		$\nabla \hat{F}_{i}(x^{k})\leftarrow \frac{1}{\alpha^{k}_{i}}\nabla F_{i}(x^{k})$,\ $i\in[m]$\\
		$\lambda^{k}\leftarrow \mathop{\arg\min}\limits_{\lambda\in\Delta_{m}}\frac{1}{2} \|\sum\limits_{i\in[m]}\lambda_{i}\nabla \hat{F}_{i}(x^{k})\|^{2}$\\  
		$d^{k}_{BB}\leftarrow-\sum\limits_{i\in[m]}\lambda_{i}^{k}\nabla \hat{F}_{i}(x^{k})$\\
		\eIf{$d^{k}_{BB}=0$}{~\\ {\bf{return}} Pareto critical point $x^{k}$ }{$\beta^{k}\leftarrow$ {\ttfamily line\_search}$(x^{k},d^{k}_{BB},JF(x^{k}))$\\
			$x^{k+1}\leftarrow x^{k}+\beta^{k}d^{k}_{BB}$}}  
	\label{alg5}
\end{algorithm}
\begin{rem}
	Comparing the lower and upper bounds of stepsize obtained by SDMO and BBDMO. If objective functions are not ill-conditioned, BBDMO will achieve a relatively large stepsize. However, it is not valid for SDMO. A direct counter-example can be Example \ref{exam1}. In which, the condition numbers of $\nabla^{2}f_{1}(x^{k})$ and $\nabla^{2}f_{2}(x^{k})$ are 1, but the stepsize is relatively small. 
\end{rem}
\subsection{Global Convergence}
In Algorithm \ref{alg5}, we see that BBDMO terminates with a Pareto critical point in a finite number of iterations or generates an infinite sequence. We will suppose that BBDMO produces an infinite sequence of noncritical points in the sequel. The main goal of this subsection is to show the convergence property of BBDMO with different line search techniques.
\begin{theo}
	Assume that $\Omega=\{x:F(x)\leq F(x^{0})\}$ is a bounded set. Let $\{x^{k}\}$ be the sequence generated by Algorithm \ref{alg5}. Then every accumulation point of $\{x^{k}\}$ is Pareto critical point.
\end{theo}
\begin{proof}
	Since all assertions in Lemma \ref{l4.1} is valid, if Armijo line search is used in Algorithm \ref{alg5}, then the proof is similar to \citep[Theorem 4.1]{FS2000}. Instead, if max-type or average-type nonmonotone line search is conducted in Algorithm \ref{alg5}. Denote $\hat{\theta}_{sd}(x^{k})$ the optimal value of (\ref{E4.1}) at $x=x^{k}$. From (\ref{E3.9}) and (\ref{E3.10}), we conclude that
	$$\max\limits_{i\in[m]}\{ \langle \nabla \hat{F}_{i}(x^{k}),d^{k}_{BB}\rangle\}\leq-2|\hat{\theta}_{sd}(x^{k})|,$$
	and
	$$\|d^{k}_{BB}\|\leq2|\hat{\theta}_{sd}(x^{k})|.$$
	This implies that \citep[Assumption 5]{MF2019} holds with $\Gamma_{1}=2$ and $\Gamma_{2}=2$. Then the assertion can be obtained by using the same arguments as in the proof of \citep[Theorem 6]{MF2019}.
\end{proof}

\section{Numerical Results} 
In this section, we present some
numerical results and demonstrate the numerical performance of BBDMO for different problems. Some comparisons with SDMO and BBMO are presented to show the efficiency of BBDMO. All numerical experiments were implemented in Python 3.7 and executed on a personal computer equipped with Intel Core i7-11390H, 3.40 GHz processor, and 16 GB of RAM.
\par The descent direction is obtained by solving the dual problem. Due to the unit simplex constraint, the dual problem is easy to solve by Frank-Wolfe/condition gradient method. For BBDMO, we set $\alpha_{\min}=10^{-3}$ and $\alpha_{\max}=10^{3}$ in (\ref{alpha_k}). In line search, we set $\sigma=0.1$, $\gamma=0.5$. Besides, we set $M=10$ in max-type nonmonotone line search, $\eta=0.8,\ q^{0}=1$ in average-type nonmonotone line search. Recall the case $\alpha^{k}\leq0$ is not considered in BBMO , then we apply the similar strategy as in (\ref{alpha_k}). In order to guarantee algorithms terminate after finite iterations, we use the stopping criterion $\|d\|<10^{-4}$ in all tested algorithms. The maximum number of iterations is set to 500. The tested algorithms are executed on several test problems, and problem illustration is given in Table \ref{tab1}. The dimension of variables and objective functions are presented in the second and third columns, respectively. $x_{L}$ and $x_{U}$ represent lower bounds and upper bounds of variables, respectively. Each problem is computed 200 times with the same initial points for different tested algorithms. Initial points are randomly selected in the internals of given lower bounds and upper bounds. The averages of 200 runs record the number of iterations, number of function evaluations, CPU time and stepsize. 
\begin{table}[t]
	\centering
	\resizebox{.65\columnwidth}{!}{
		\begin{tabular}{llllll}
			\hline
			Problem   & $n$   & $m$ & $x_{L}$                & $x_{U}$               & Reference \\ \hline
			Imbalance1 & 2   & 2 & {[}-2,-2{]}       & {[}2,2{]}        &     --      \\
			Imbalance2 & 2   & 2 & {[}-2,-2{]}       & {[}2,2{]}        &    --       \\
			JOS1a     & 50  & 2 & -{[}2,...,2{]}    & {[}2,...,2{]}    &     \citep{JO2001}      \\
			JOS1b     & 100 & 2 & -{[}2,...,2{]}    & {[}2,...,2{]}    &      \citep{JO2001}      \\
			JOS1c     & 100 & 2 & -50{[}1,...,1{]}  & 50{[}1,...,1{]}  &     \citep{JO2001}       \\
			JOS1d     & 100 & 2 & -100{[}1,...,1{]} & 100{[}1,...,1{]} &    \citep{JO2001}       \\
			WIT1      & 2   & 2 & {[}-2,-2{]}       & {[}2,2{]}        &    \citep{W2012}        \\
			WIT2      & 2   & 2 & {[}-2,-2{]}       & {[}2,2{]}        &       \citep{W2012}       \\
			WIT3      & 2   & 2 & {[}-2,-2{]}       & {[}2,2{]}        &    \citep{W2012}          \\
			WIT4      & 2   & 2 & {[}-2,-2{]}       & {[}2,2{]}        &     \citep{W2012}         \\
			WIT5      & 2   & 2 & {[}-2,-2{]}       & {[}2,2{]}        &     \citep{W2012}         \\
			WIT6      & 2   & 2 & {[}-2,-2{]}       & {[}2,2{]}        &     \citep{W2012}         \\
			Deb       & 2   & 2 & {[}0.1,0.1{]}     & {[}1,1{]}        &    \citep{D1999}          \\
			PNR       & 2   & 2 & {[}-2,-2{]}       & {[}2,2{]}        &    \citep{PN2006}          \\
			DD1       & 5   & 2 & -20{[}1,...,1{]}  & 20{[}1,...,1{]}  &     \citep{DD1998}         \\
			FDS       & 10  & 3 & -{[}2,...,2{]}    & {[}2,...,2{]}    &    \citep{FD2009}          \\
			TRIDIA1   & 3   & 3 & -{[}1,...,1{]}    & {[}1,...,1{]}    &     \citep{T1983}         \\
			TRIDIA2   & 4   & 4 & -{[}1,...,1{]}    & {[}1,...,1{]}    &      \citep{T1983}     \\ \hline
		\end{tabular}	
	}
	\caption{Description of all test problems used in numerical experiments.}
	\label{tab1}
\end{table}
\begin{prob}
	Consider the following multiobjective optimization:
	$$(\mathrm{Imbalance})\ \min\limits_{x_{1},x_{2}\in[-2,2]}(ax_{1}^{2}+bx_{2}^{2},c(x_{1}-50)^{2}+d(x_{2}+50)^{2})$$
	where Imbalance1 with $a=0.1,b=10,c=1,d=100$, Imbalance2 with $a,b=1,c,d=100$.
\end{prob}

\begin{prob}
	Consider the following multiobjective optimization:
	$$(\mathrm{JOS1})\ \min\limits_{x\in \mathbb{R}^{n}}(\frac{1}{n}\sum\limits_{i=1}^{n}x_{i}^{2},\frac{1}{n}\sum\limits_{i=1}^{n}(x_{i}-2)^{2}).$$
\end{prob}
\begin{prob}
	Consider the following parametric multiobjective optimization:
	$$(\mathrm{WIT})\ \min\limits_{x\in \mathbb{R}^{2}}(f_{1}(x,\lambda),f_{2}(x,\lambda))$$
	where $$f_{1}(x,\lambda)=\lambda((x_{1}-2)^{2}+(x_{2}-2)^{2})+(1-\lambda)((x_{1}-2)^{4}+(x_{2}-2)^{8}),$$ $$f_{2}(x,\lambda)=(x_{1}+2\lambda)^{2}+(x_{2}+2\lambda)^{2}.$$ In which, $\lambda=0,0.5,0.9,0.99,0.999,1$ represent $\mathrm{WIT1}$-$\mathrm{6}$, respectively.
\end{prob}

\begin{figure}[H]
	\centering
	\centering
	\subfigure[SDMO]{
		\includegraphics[scale=0.3]{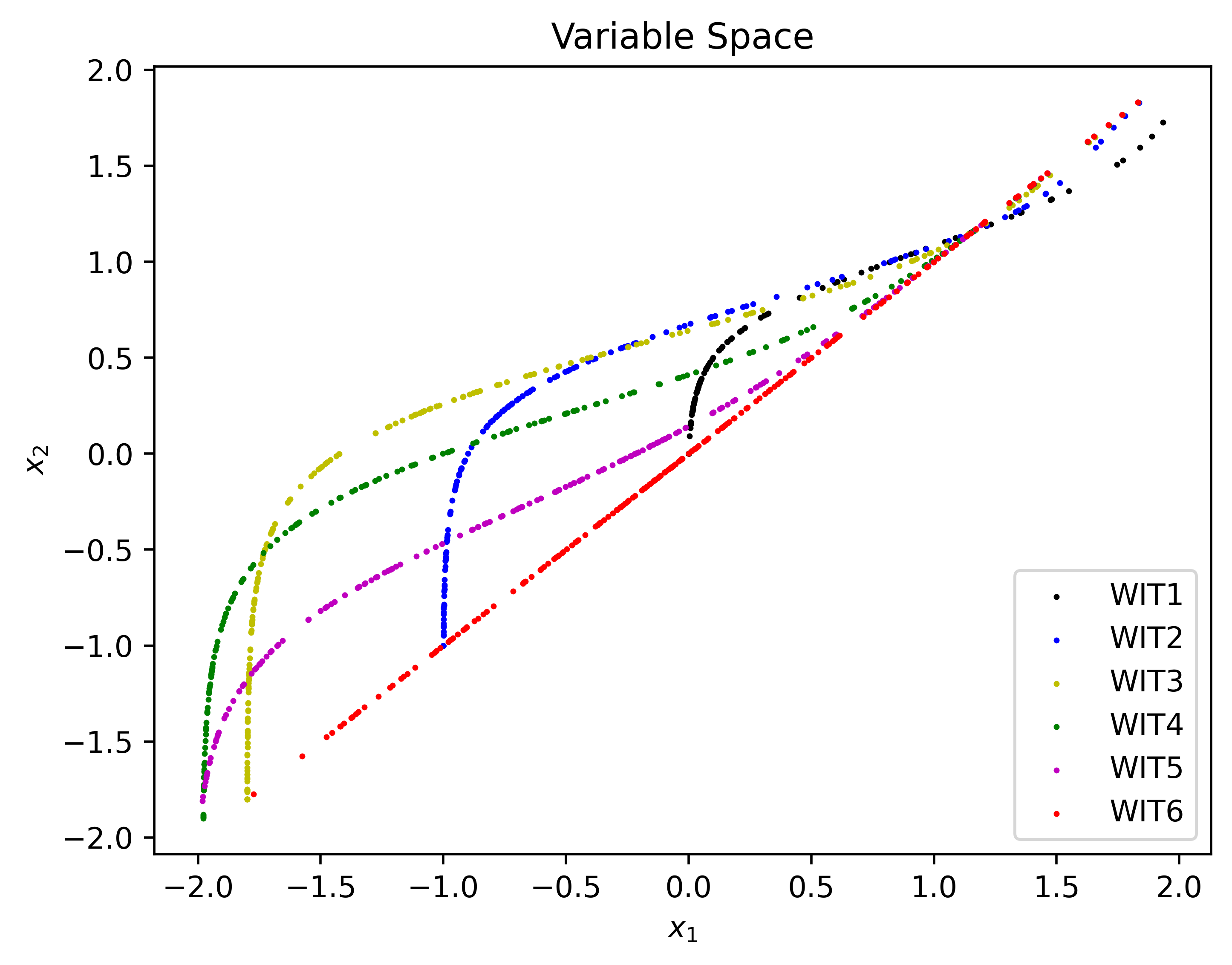} 
	}
	\subfigure[BBMO]{
		\includegraphics[scale=0.3]{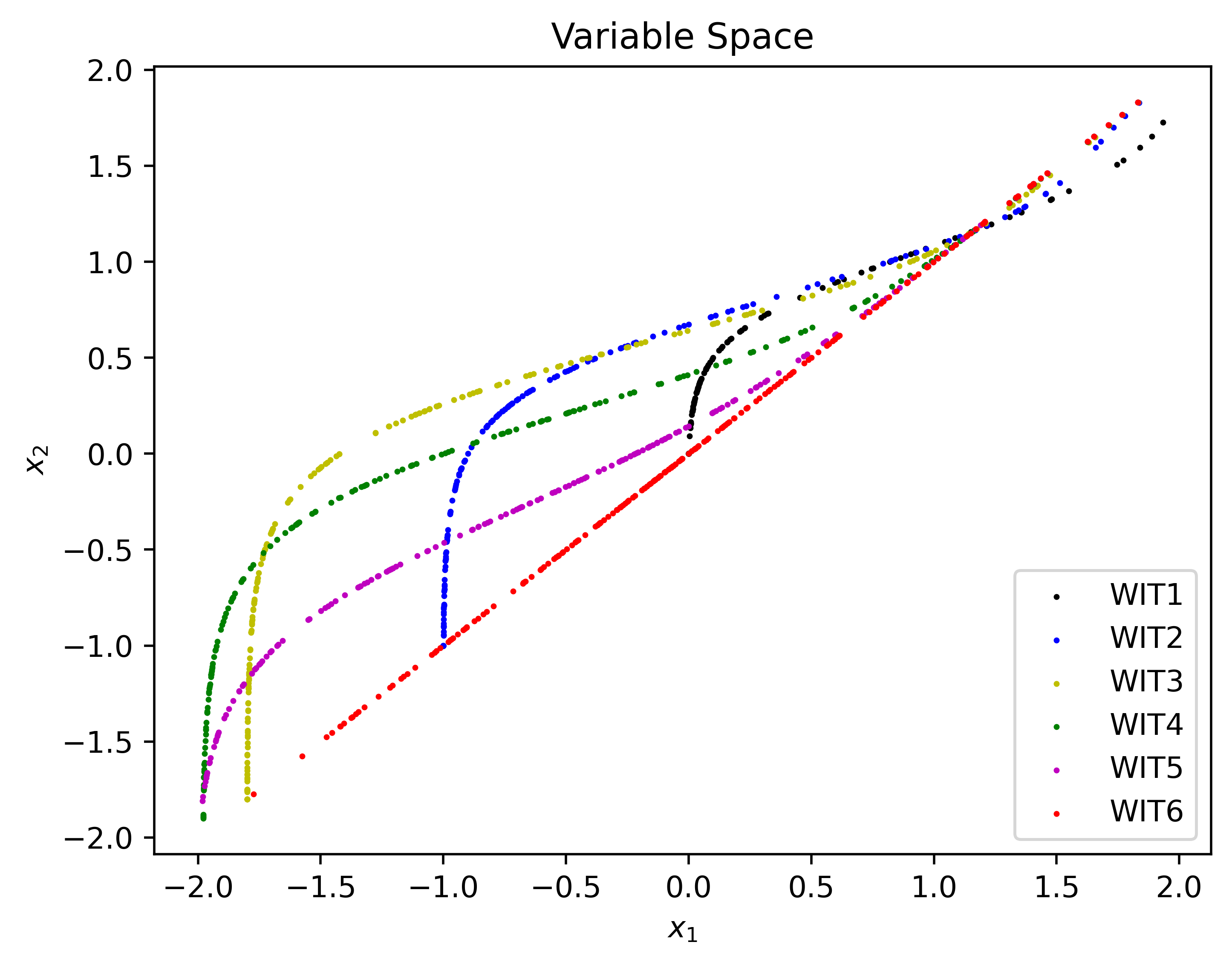} 
	}
	\subfigure[BBDMO]{
		\includegraphics[scale=0.3]{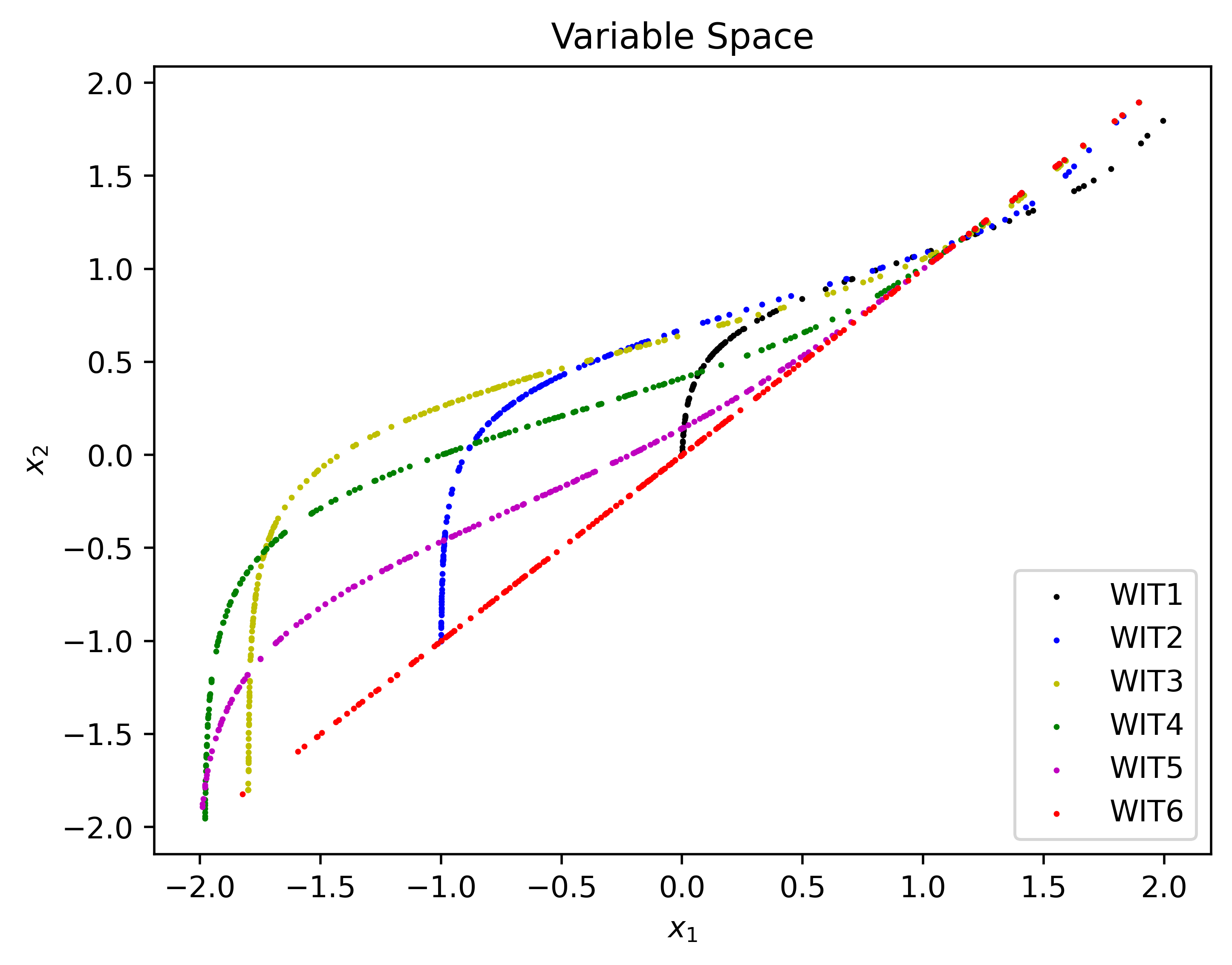} 
	}
	\caption{Numerical results in variable space obtained by SDMO, BBMO, BBDMO with monotone line search for problems WIT1-6.}\label{wit}	
\end{figure}

\begin{prob}
	Consider the following multiobjective optimization:
	$$(\mathrm{Deb})\ \min\limits_{x_{1}>0}(x_{1},\frac{g(x_{2})}{x_{1}})$$
	where $g(x_{2})=2-exp\{-(\frac{x_{2}-0.2}{0.004})^{2}\}-0.8exp\{-(\frac{x_{2}-0.6}{0.4})^{2}\}$.
\end{prob}
\begin{figure}[H]
	\centering
	\centering
	\subfigure[SDMO]{
		\includegraphics[scale=0.3]{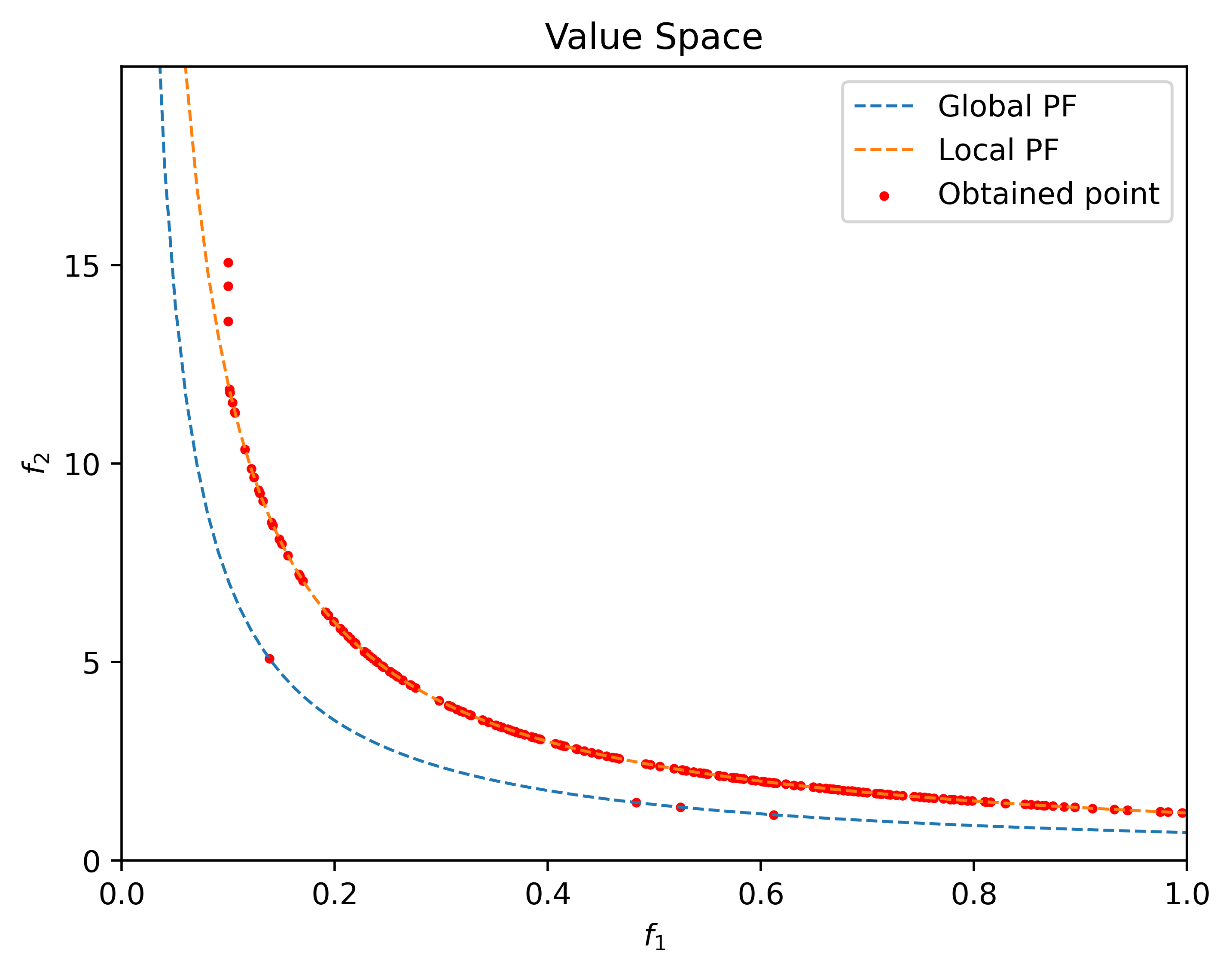} 
	}
	\subfigure[BBMO]{
		\includegraphics[scale=0.3]{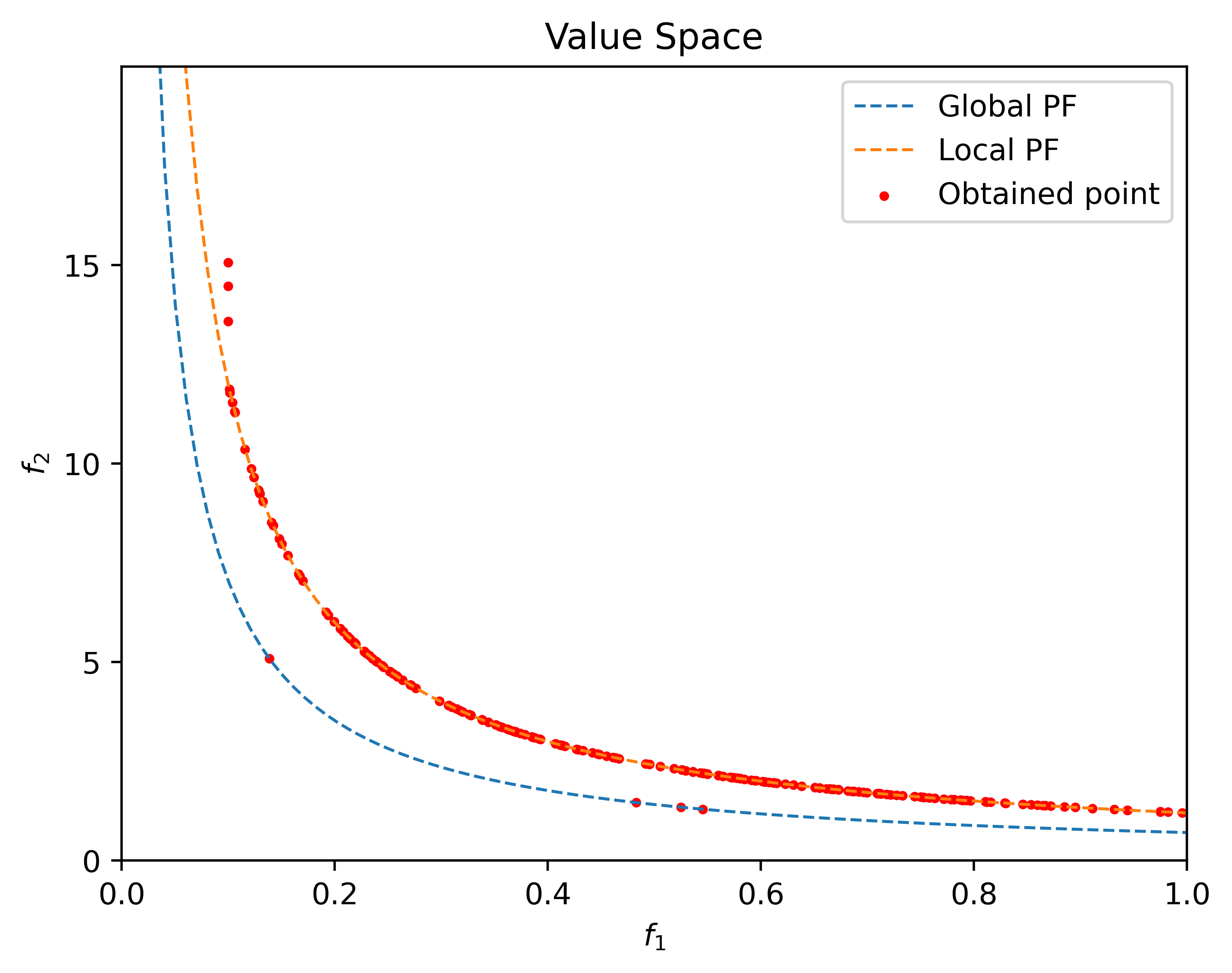} 
	}
	\subfigure[BBDMO]{
		\includegraphics[scale=0.3]{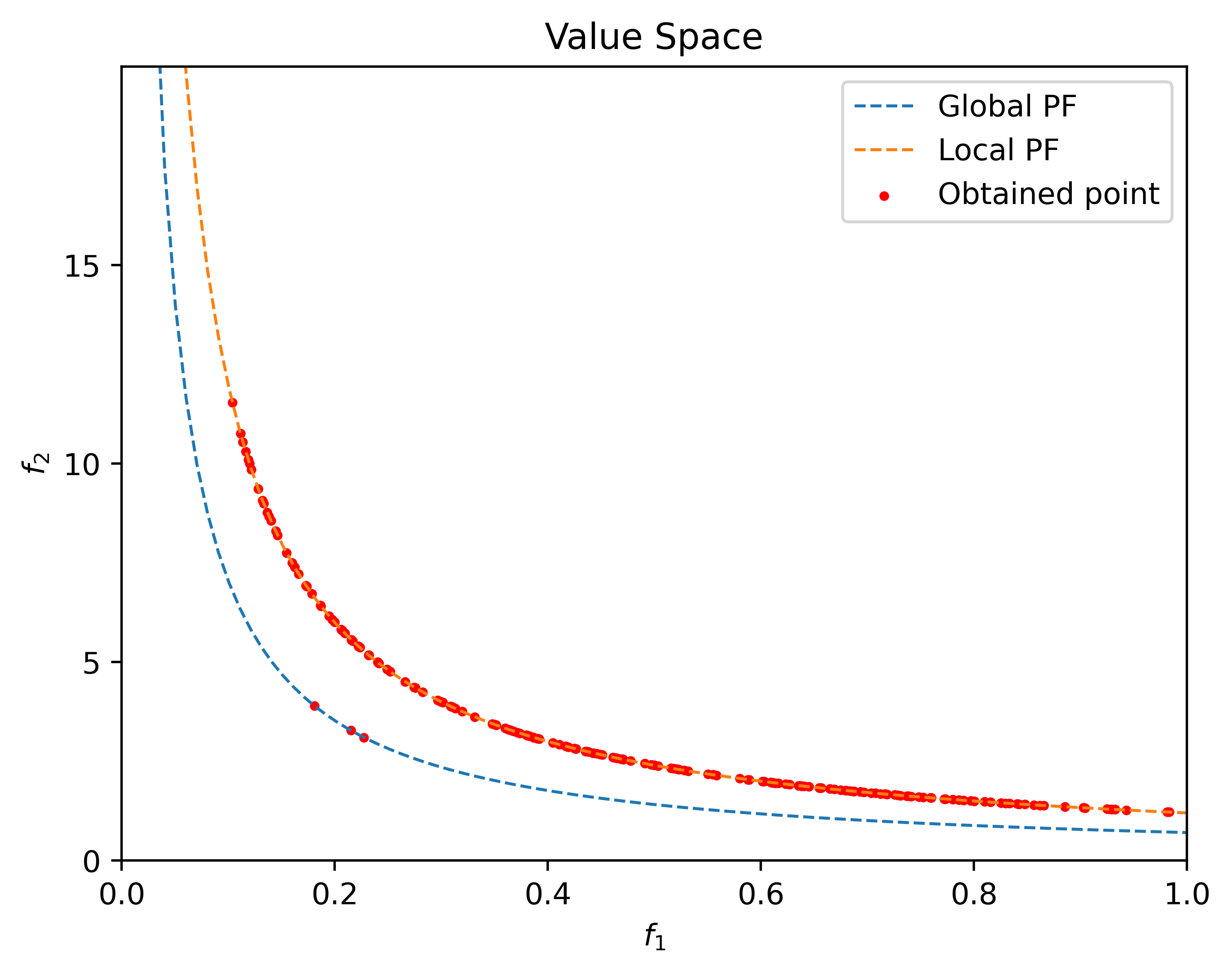} 
	}
	\caption{(a) Local and global minimizers of $g(x_{2})$. (b-d) Numerical results in value space obtained by SDMO, BBMO, and BBDMO with monotone line search for problem Deb.}	
\end{figure}
\begin{prob}
	Consider the following multiobjective optimization:
	$$(\mathrm{PNR})\ \min\limits_{x\in \mathbb{R}^{2}}(f_{1}(x),f_{2}(x))$$
	where $$f_{1}(x)=x_{1}^{4}+x_{2}^{4}-x_{1}^{2}+x_{2}^{2}-10x_{1}x_{2}+0.25x_{1}+20,$$ $$f_{2}(x)=(x_{1}-1)^{2}+x_{2}^{2}.$$
\end{prob}
\begin{figure}[H]
	\centering
	\centering
	\subfigure[SDMO]{
		\includegraphics[scale=0.3]{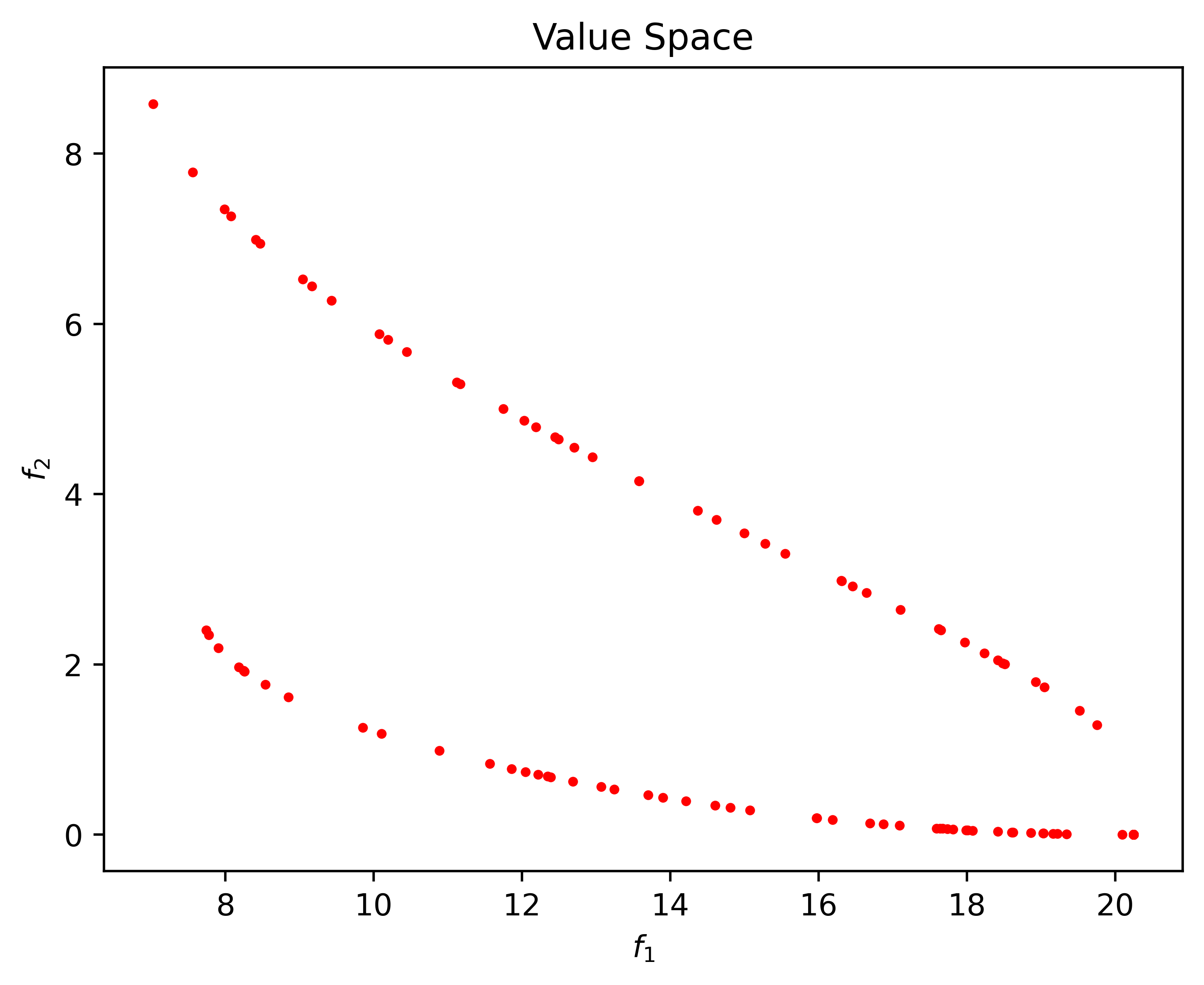} 
	}
	\subfigure[BBMO]{
		\includegraphics[scale=0.3]{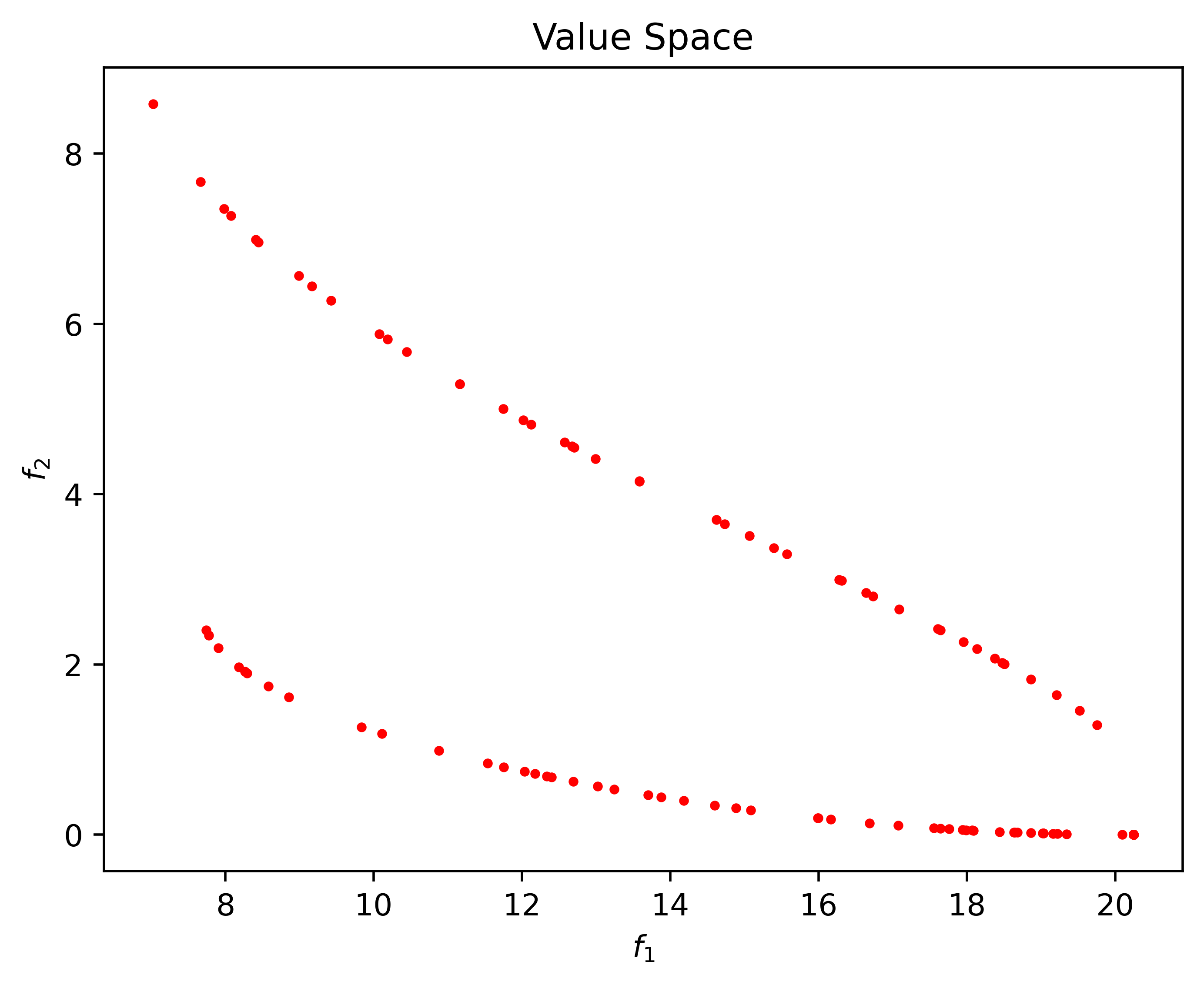} 
	}
	\subfigure[BBDMO]{
		\includegraphics[scale=0.3]{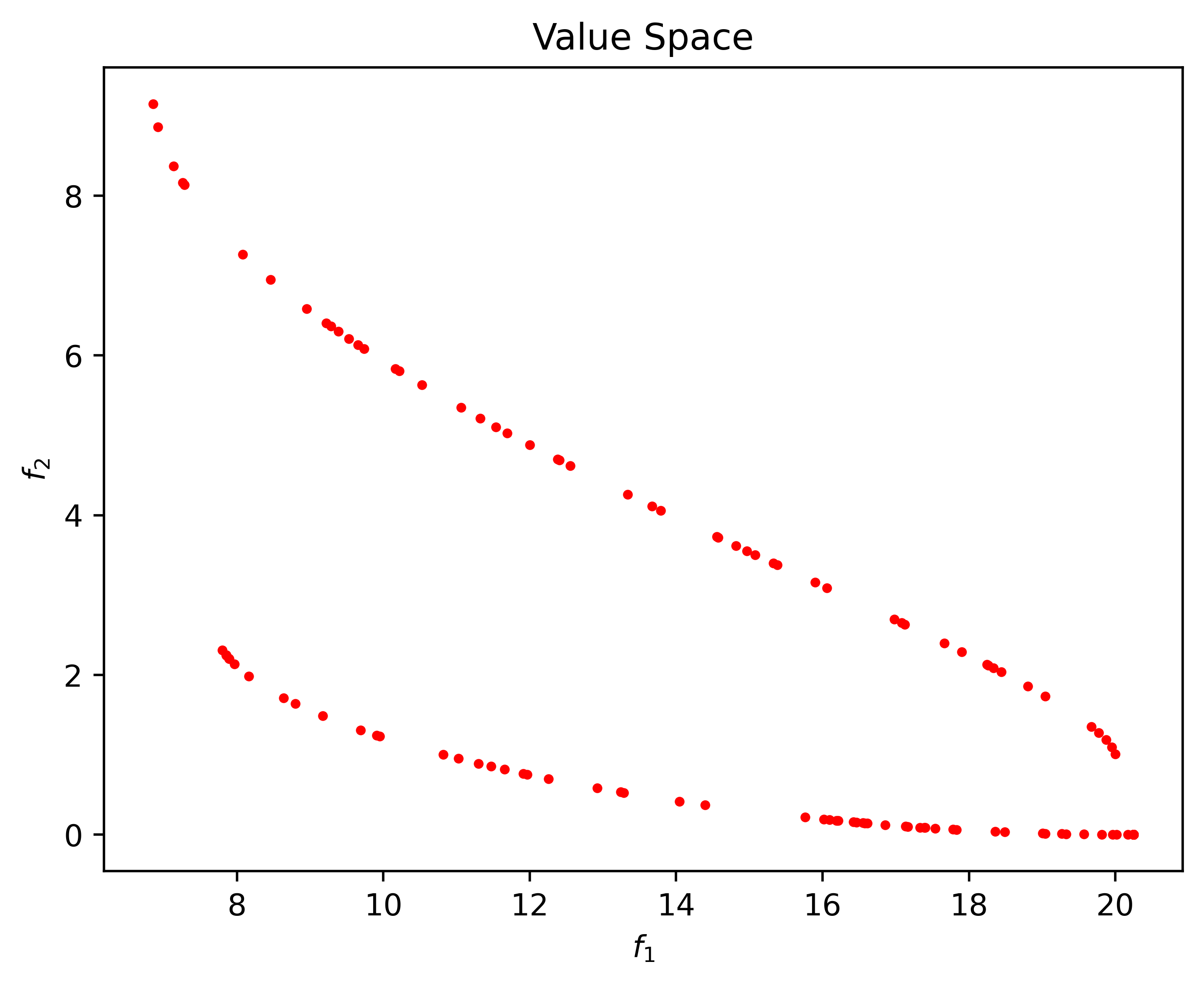} 
	}
	\caption{Numerical results in value space obtained by SDMO, BBMO, and BBDMO with monotone line search for problem PNR.}\label{pnr}	
\end{figure}
\begin{prob}
	Consider the following multiobjective optimization:
	$$(\mathrm{DD1})\ \min\limits_{x\in \mathbb{R}^{5}}(f_{1}(x),f_{2}(x))$$
	where $$f_{1}(x)=\sum\limits_{i=1}^{5}x^{2}_{i},$$ $$f_{2}(x)=3x_{1}+2x_{2}-\frac{x_{3}}{3}+0.01(x_{4}-x_{5})^{3}.$$
\end{prob}
\begin{figure}[H]
	\centering
	\centering
	\subfigure[SDMO]{
		\includegraphics[scale=0.3]{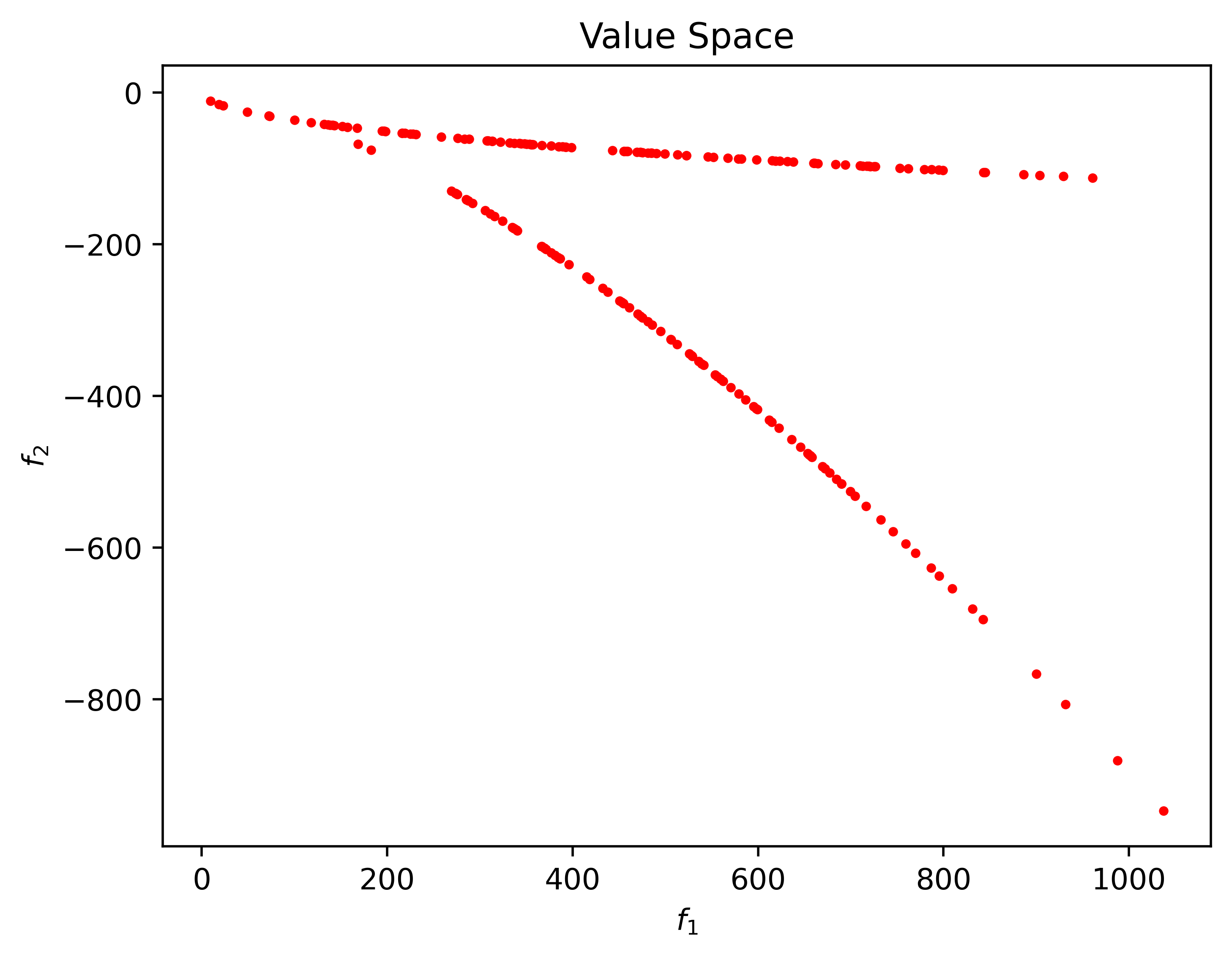} 
	}
	\subfigure[BBMO]{
		\includegraphics[scale=0.3]{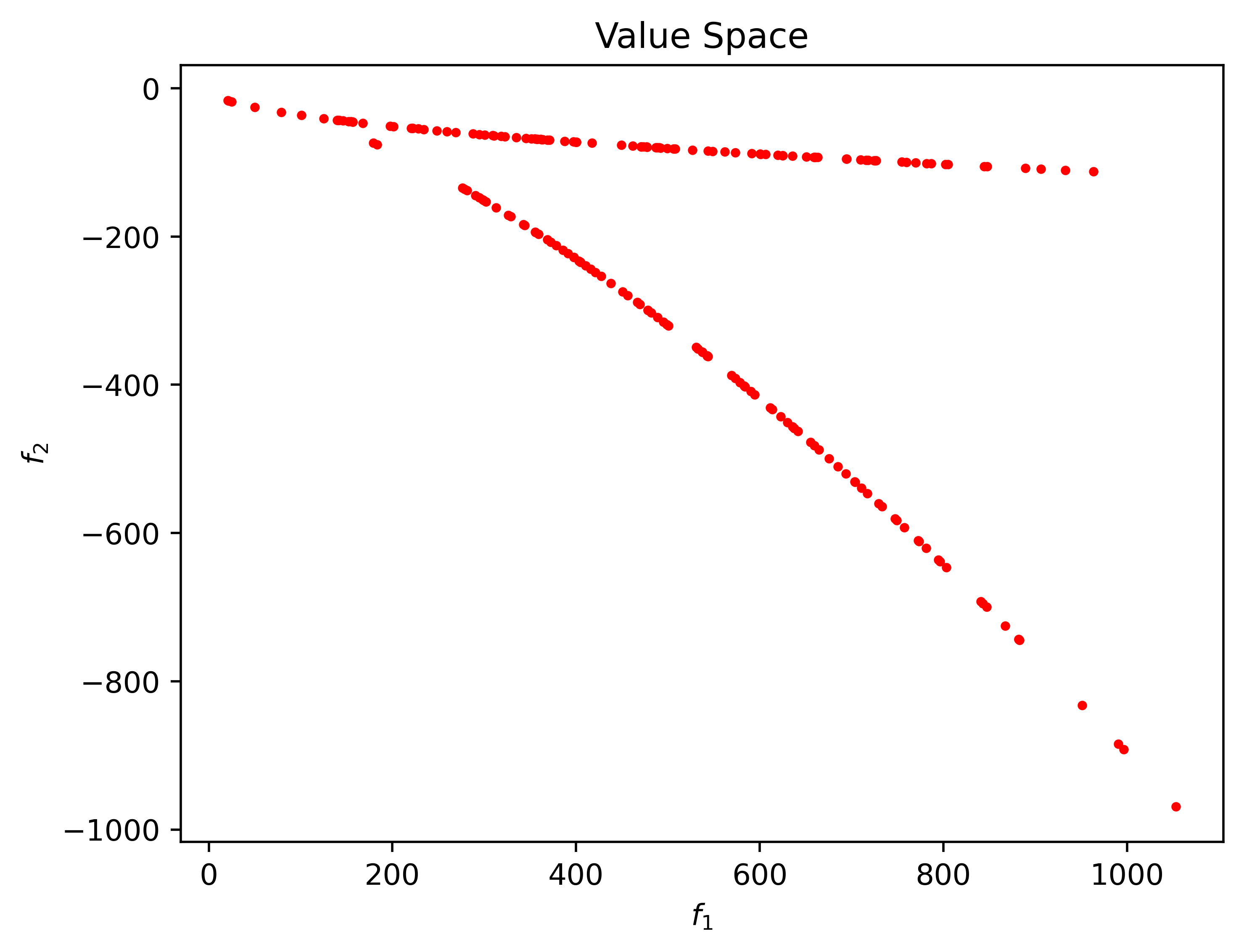} 
	}
	\subfigure[BBDMO]{
		\includegraphics[scale=0.3]{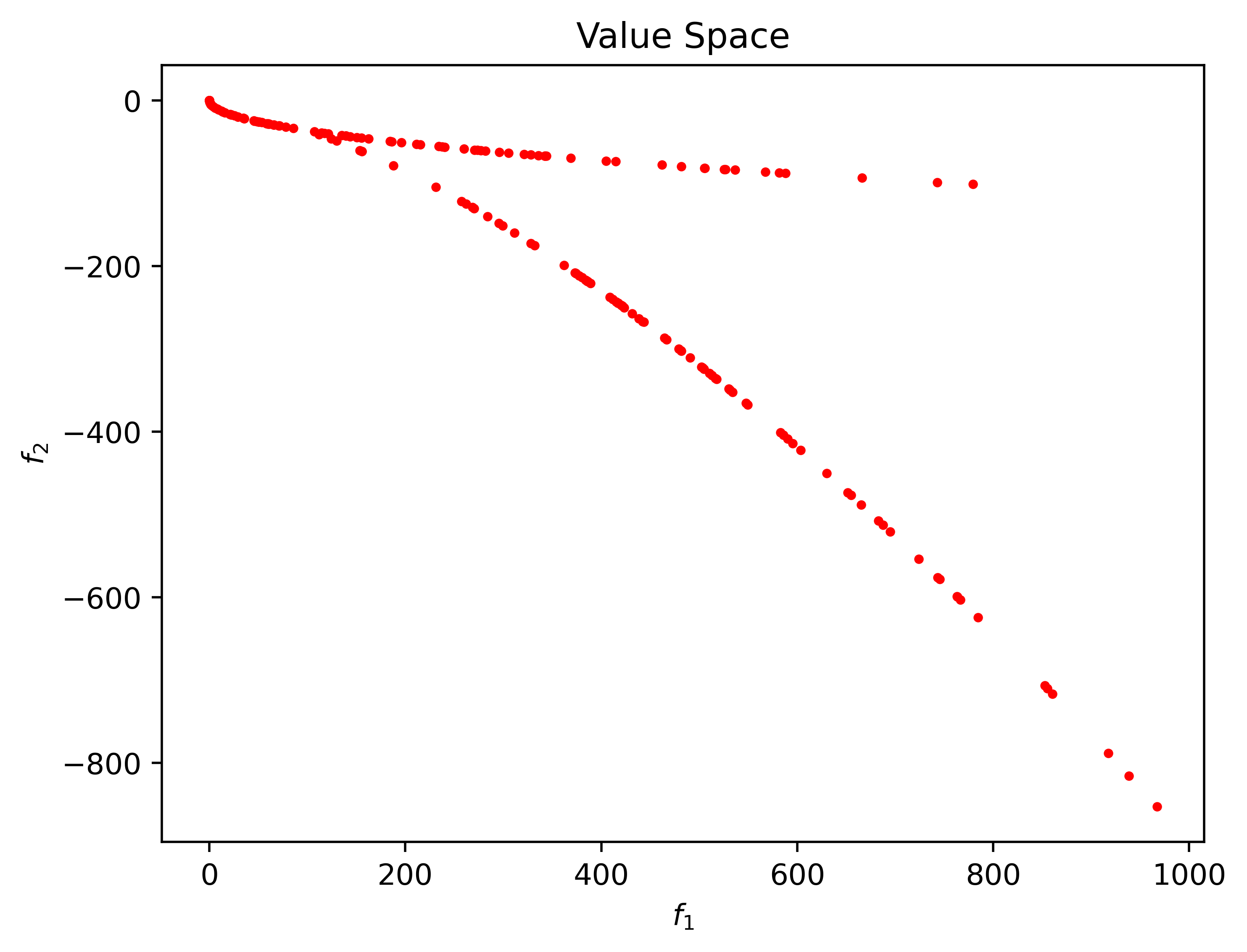} 
	}
	\caption{Numerical results in value space obtained by SDMO, BBMO, and BBDMO with monotone line search for problem DD1.}\label{DD1}	
\end{figure}
\begin{prob}
	Consider the following multiobjective optimization:
	$$(\mathrm{FDS})\ \min\limits_{x\in \mathbb{R}^{n}}(f_{1}(x),f_{2}(x),f_{3}(x))$$
	where $$f_{1}(x)=\frac{1}{n}\sum\limits_{i=1}^{n}i(x_{i}-i)^{2},$$ $$f_{2}(x)=\mathrm{exp}(\sum\limits_{i=1}^{n}\frac{x_{i}}{n})+\|x\|^{2}_{2},$$
	$$f_{3}(x)=\frac{1}{n(n+1)}\sum\limits_{i=1}^{n}i(n-i+1)e^{-x_{i}}.$$
\end{prob}
\begin{figure}[H]
	\centering
	\centering
	\subfigure[SDMO]{
		\includegraphics[scale=0.4]{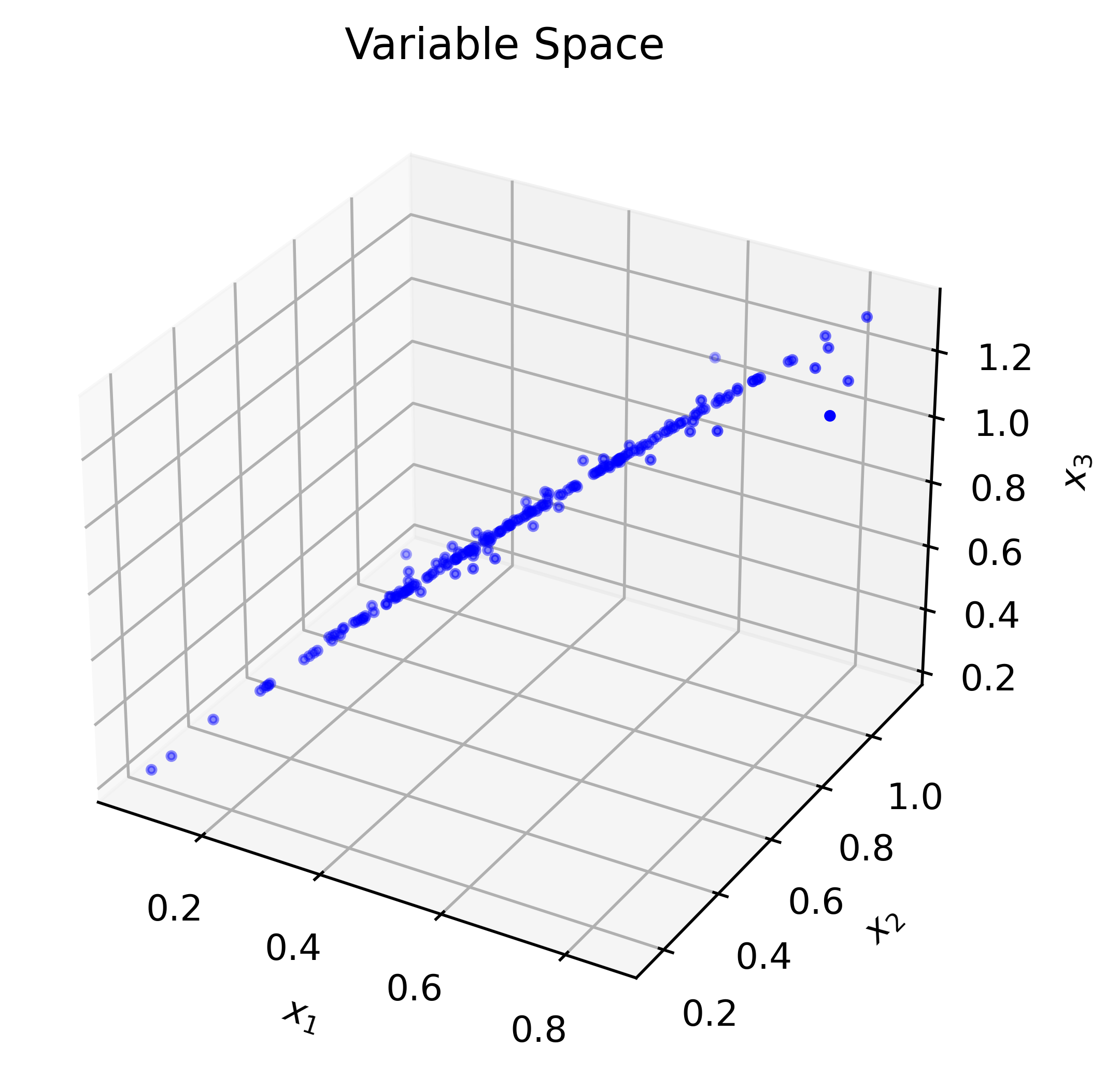} 
	}
	\subfigure[BBMO]{
		\includegraphics[scale=0.4]{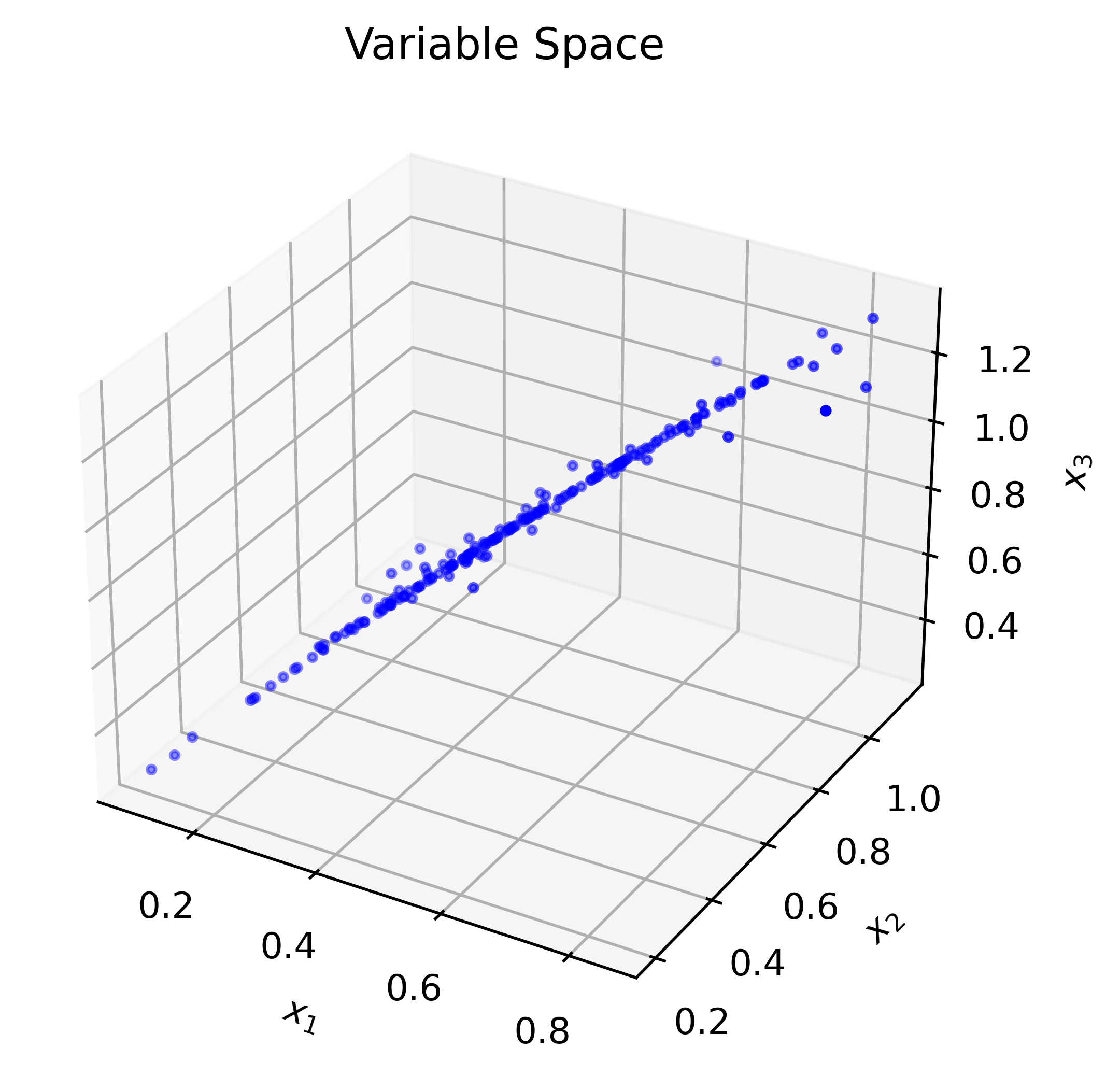} 
	}
	\subfigure[BBDMO]{
		\includegraphics[scale=0.4]{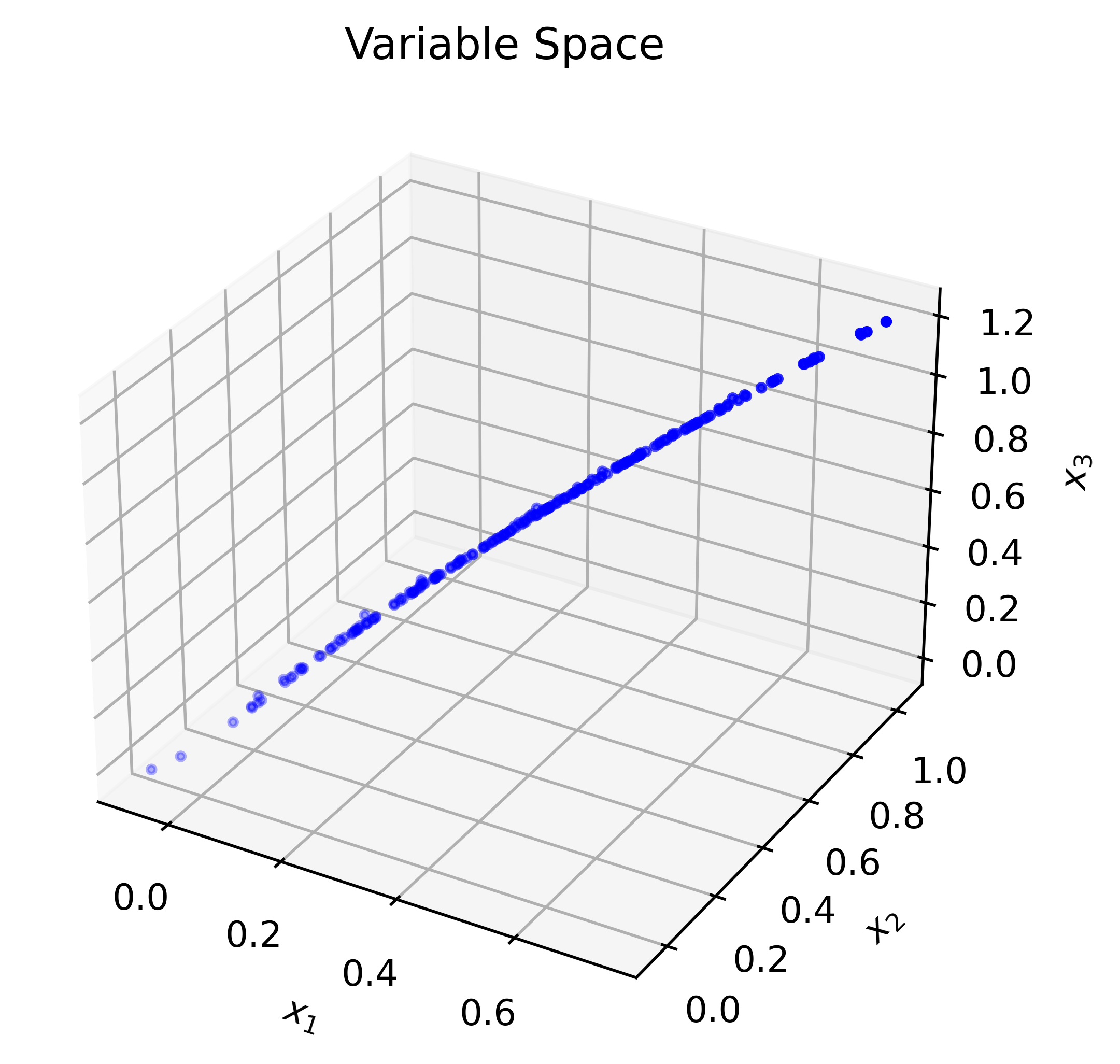} 
	}
	\caption{Numerical results in $(x_{1},x_{2},x_{3})$ variable space obtained by SDMO, BBMO, and BBDMO with monotone line search for problem FDS.}\label{FDS}	
\end{figure}

\begin{prob}
	Consider the following multiobjective optimization:
	$$(\mathrm{TRIDIA1})\ \min\limits_{x\in \mathbb{R}^{3}}(f_{1}(x),f_{2}(x),f_{3}(x))$$
	where $$f_{1}(x)=(2x_{1}-1)^{2},$$ $$f_{2}(x)=2(2x_{1}-x_{2})^{2},$$
	$$f_{3}(x)=3(x_{2}-x_{3})^{2}.$$
\end{prob}

\begin{prob}
	Consider the following multiobjective optimization:
	$$(\mathrm{TRIDIA2})\ \min\limits_{x\in \mathbb{R}^{4}}(f_{1}(x),f_{2}(x),f_{3}(x),f_{4}(x))$$
	where $$f_{1}(x)=(2x_{1}-1)^{2}+x^{2}_{2},$$ $$f_{i}(x)=i(2x_{i-1}-x_{i})^{2}-(i-1)x^{2}_{i-1}+ix_{i}^{2},\ i=2,3,$$
	$$f_{4}(x)=4(2x_{3}-x_{4})^{2}-3x^{2}_{3}.$$
\end{prob}

\begin{table}[t]
	\centering
	\resizebox{.9\columnwidth}{!}{
		\begin{tabular}{llllllllllllllll}
			\hline
			\multicolumn{1}{l}{Problem} &                      & \multicolumn{4}{l}{SDMO}                                 &                      & \multicolumn{4}{l}{BBMO}                                 &                      & \multicolumn{4}{l}{OBBMO}                                  \\ \cline{3-6} \cline{8-11} \cline{13-16} 
			\multicolumn{1}{l}{}        &                      & iter          & feval         & time($ms$)          & stepsize & \multicolumn{1}{c}{} & iter          & feval         & time($ms$)          & stepsize & \multicolumn{1}{c}{} & iter           & feval          & time($ms$)          & stepsize \\ \hline
			Imbalance1                  &                      & 68.93         & 139.26        & 9.54          & 0.49     &                      & 56.10         & 222.37        & 11.47         & 0.60     &                      & \textbf{2.88}  & \textbf{5.38}  & \textbf{0.61} & 0.75     \\
			Imbalance2                  &                      & 241.55        & 1930.08       & 68.05         & 0.01     &                      & 242.67        & 1697.30       & 68.00         & 0.01     &                      & \textbf{1.49}  & \textbf{5.36}  & \textbf{0.43} & 0.51     \\
			JOS1a                       &                      & 198.35        & 198.35        & 25.63         & 1.00     &                      & \textbf{2.00} & \textbf{2.00} & \textbf{0.41} & 13.00    &                      & \textbf{2.00}  & \textbf{2.00}  & 0.51          & 1.00     \\
			JOS1b                       &                      & 383.44        & 383.44        & 59.21         & 1.00     &                      & \textbf{2.00} & \textbf{2.00} & \textbf{0.43} & 25.50    &                      & \textbf{2.00}  & \textbf{2.00}  & 0.50          & 1.00     \\
			JOS1c                       &                      & 500.00        & 500.00        & 75.18         & 1.00     &                      & \textbf{2.00} & \textbf{2.00} & \textbf{0.42} & 25.50    &                      & \textbf{2.00}  & \textbf{2.00}  & 0.51          & 1.00     \\
			JOS1d                       &                      & 500.00        & 500.00        & 70.33         & 1.00     &                      & \textbf{2.00} & \textbf{2.00} & \textbf{0.37} & 25.50    &                      & \textbf{2.00}  & \textbf{2.00}  & 0.54          & 1.00     \\
			WIT1                        &                      & 73.59         & 494.22        & 20.18         & 0.03     &                      & 60.40         & 317.06        & 14.85         & 0.04     &                      & \textbf{2.57}  & \textbf{4.83}  & \textbf{0.47} & 0.70     \\
			WIT2                        &                      & 108.75        & 735.53        & 29.78         & 0.02     &                      & 105.48        & 565.01        & 25.87         & 0.02     &                      & \textbf{3.35}  & \textbf{6.75}  & \textbf{0.82} & 0.69     \\
			WIT3                        &                      & 49.14         & 249.29        & 10.77         & 0.06     &                      & 47.85         & 181.03        & 9.81          & 0.06     &                      & \textbf{3.68}  & \textbf{6.16}  & \textbf{0.93} & 0.73     \\
			WIT4                        &                      & 10.65         & 35.38         & 2.02          & 0.21     &                      & 9.14          & 17.61         & 1.49          & 0.22     &                      & \textbf{3.32}  & \textbf{4.67}  & \textbf{0.75} & 0.78     \\
			WIT5                        &                      & 7.08          & 18.58         & 1.21          & 0.33     &                      & 6.59          & 11.80         & 1.06          & 0.28     &                      & \textbf{3.22}  & \textbf{4.33}  & \textbf{0.67} & 0.81     \\
			WIT6                        &                      & \textbf{1.00} & \textbf{2.00} & 0.32          & 0.50     &                      & \textbf{1.00} & \textbf{2.00} & 0.34          & 0.50     &                      & \textbf{1.00}  & \textbf{2.00}  & \textbf{0.31} & 0.50     \\
			Deb                         &                      & 95.62         & 625.88        & 29.29         & 0.03     &                      & 93.67         & 546.01        & 27.83         & 0.03     &                      & \textbf{3.91}  & \textbf{7.09}  & \textbf{0.87} & 0.73     \\
			PNR                         &                      & 9.84          & 44.00         & 2.07          & 0.11     &                      & 10.51         & 27.28         & 1.88          & 0.11     &                      & \textbf{2.70}  & \textbf{4.60}  & \textbf{0.63} & 0.73     \\
			DD                          & \multicolumn{1}{c}{} & 77.93         & 154.97        & 11.87         & 0.51     &                      & 65.50         & 270.31        & 15.20         & 0.60     &                      & \textbf{7.44}  & \textbf{7.99}  & \textbf{1.52} & 0.96     \\
			FDS                         & \multicolumn{1}{c}{} & 246.54        & 1145.91       & 402.24        & 0.17     &                      & 234.77        & 1692.15       & 541.25        & 0.18     &                      & \textbf{6.74}  & \textbf{7.30}  & \textbf{0.59} & 0.97     \\
			TRIDIA\_1                   & \multicolumn{1}{c}{} & 5.47          & 12.04         & \textbf{2.14} & 0.59     &                      & 5.65          & 10.05         & 2.34          & 0.49     &                      & \textbf{4.28}  & \textbf{7.53}  & 2.68          & 0.65     \\
			TRIDIA\_2                   & \multicolumn{1}{c}{} & 27.73         & 181.88        & 37.10         & 0.11     &                      & 21.30         & 131.48        & 21.20         & 0.11     &                      & \textbf{10.15} & \textbf{18.95} & \textbf{7.62} & 0.62     \\ \hline
		\end{tabular}
	}
	\caption{Number of average iterations (iter), number of average function evaluations (feval), average CPU time (time($ms$)) and average stepsize (stepsize) of SDMO, BBMO, and BBDMO implemented on different test problems with \textbf{monotone} line search.}
	\label{tab2}
\end{table}

\begin{table}[t]
	\centering
	\resizebox{.9\columnwidth}{!}{
		\begin{tabular}{llllllllllllllll}
			\hline
			\multicolumn{1}{l}{Problem} &                      & \multicolumn{4}{l}{SDMO}                                                                            &                      & \multicolumn{4}{l}{BBMO}                                  &                      & \multicolumn{4}{l}{OBBMO}                                 \\ \cline{3-6} \cline{8-11} \cline{13-16} 
			\multicolumn{1}{l}{}        &                      & \multicolumn{1}{l}{iter} & \multicolumn{1}{l}{feval} & time($ms$)          & \multicolumn{1}{l}{stepsize} & \multicolumn{1}{c}{} & iter          & feval          & time($ms$)          & stepsize & \multicolumn{1}{c}{} & iter          & feval          & time($ms$)          & stepsize \\ \hline
			Imbalance1                  &                      & 34.76                    & 39.96                     & 4.34          & 0.94                         &                      & 15.79         & 47.39          & 3.07          & 1.23     &                      & \textbf{2.88} & \textbf{5.38}  & \textbf{0.72} & 0.75     \\
			Imbalance2                  &                      & 232.21                   & 1822.26                   & 61.49         & 0.01                         &                      & 232.03        & 1587.21        & 64.98         & 0.01     &                      & \textbf{1.49} & \textbf{5.36}  & \textbf{0.36} & 0.51     \\
			JOS1a                       &                      & 198.35                   & 198.35                    & 28.75         & 1.00                         &                      & \textbf{2.00} & \textbf{2.00}  & \textbf{0.53} & 13.00    &                      & \textbf{2.00} & \textbf{2.00}  & 0.54          & 1.00     \\
			JOS1b                       &                      & 383.44                   & 383.44                    & 63.76         & 1.00                         &                      & \textbf{2.00} & \textbf{2.00}  & \textbf{0.50} & 25.50    &                      & \textbf{2.00} & \textbf{2.00}  & 0.54          & 1.00     \\
			JOS1c                       &                      & 500.00                   & 500.00                    & 79.81         & 1.00                         &                      & \textbf{2.00} & \textbf{2.00}  & 0.53          & 25.50    &                      & \textbf{2.00} & \textbf{2.00}  & \textbf{0.48} & 1.00     \\
			JOS1d                       &                      & 500.00                   & 500.00                    & 73.98         & 1.00                         &                      & \textbf{2.00} & \textbf{2.00}  & \textbf{0.49} & 25.50    &                      & \textbf{2.00} & \textbf{2.00}  & 0.51          & 1.00     \\
			WIT1                        &                      & 34.12                    & 183.82                    & 7.36          & 0.16                         &                      & 20.41         & 92.37          & 4.96          & 0.10     &                      & \textbf{2.55} & \textbf{4.70}  & \textbf{0.58} & 0.72     \\
			WIT2                        &                      & 69.05                    & 402.54                    & 16.34         & 0.07                         &                      & 58.32         & 277.15         & 14.69         & 0.04     &                      & \textbf{3.33} & \textbf{6.58}  & \textbf{0.87} & 0.71     \\
			WIT3                        &                      & 29.13                    & 89.89                     & 5.04          & 0.38                         &                      & 10.63         & 28.27          & 1.99          & 0.16     &                      & \textbf{3.61} & \textbf{5.85}  & \textbf{0.94} & 0.76     \\
			WIT4                        &                      & 27.94                    & 43.19                     & 3.98          & 0.76                         &                      & 3.46          & 4.63           & \textbf{0.74} & 0.36     &                      & \textbf{3.33} & \textbf{4.58}  & 0.94          & 0.80     \\
			WIT5                        &                      & 30.82                    & 44.68                     & 4.29          & 0.79                         &                      & 3.32          & 4.35           & \textbf{0.62} & 0.42     &                      & \textbf{3.22} & \textbf{4.33}  & 0.92          & 0.81     \\
			WIT6                        &                      & \textbf{1.00}            & \textbf{2.00}             & \textbf{0.36} & 0.50                         &                      & \textbf{1.00} & \textbf{2.00}  & 0.39          & 0.50     &                      & \textbf{1.00} & \textbf{2.00}  & 0.38          & 0.50     \\
			Deb                         &                      & 65.14                    & 385.94                    & 17.99         & 0.10                         &                      & 56.54         & 330.91         & 19.37         & 0.05     &                      & \textbf{4.00} & \textbf{7.04}  & \textbf{0.94} & 0.74     \\
			PNR                         &                      & 10.20                    & 28.13                     & 1.84          & 0.42                         &                      & 2.68          & 4.68           & \textbf{0.58} & 0.26     &                      & \textbf{2.67} & \textbf{4.45}  & 0.77          & 0.74     \\
			DD                          & \multicolumn{1}{c}{} & 39.06                    & 42.99                     & 5.19          & 0.95                         &                      & 27.75         & 95.97          & 6.26          & 1.19     &                      & \textbf{7.44} & \textbf{7.99}  & \textbf{1.71} & 0.96     \\
			FDS                         & \multicolumn{1}{c}{} & 165.06                   & 699.86                    & 253.06        & 0.27                         &                      & 30.16         & 190.86         & 67.14         & 1.11     &                      & \textbf{7.34} & \textbf{7.39}  & \textbf{5.75} & 0.99     \\
			TRIDIA\_1                   & \multicolumn{1}{c}{} & 22.12                    & 44.24                     & 8.13          & 0.62                         &                      & 14.60         & 17.87          & 7.24          & 0.83     &                      & \textbf{8.37} & \textbf{10.57} & \textbf{6.91} & 0.74     \\
			TRIDIA\_2                   & \multicolumn{1}{c}{} & 32.03                    & 76.21                     & 29.84         & 0.61                         &                      & \textbf{9.27} & \textbf{11.18} & \textbf{6.34} & 0.94     &                      & 11.39         & 15.76          & 24.48         & 0.79     \\ \hline
		\end{tabular}
	}
	\caption{Number of average iterations (iter), number of average function evaluations (feval), average CPU time (time($ms$)) and average stepsize (stepsize) of SDMO, BBMO, and BBDMO implemented on different test problems with \textbf{max-type nonmonotone} line search.}
	\label{tab3}
\end{table}

\begin{table}[t]
	\centering
	\resizebox{.9\columnwidth}{!}{
		\begin{tabular}{llllllllllllllll}
			\hline
			\multicolumn{1}{l}{Problem} &                      & \multicolumn{4}{l}{SDMO}                                                                            &                      & \multicolumn{4}{l}{BBMO}                                  &                      & \multicolumn{4}{l}{OBBMO}                                 \\ \cline{3-6} \cline{8-11} \cline{13-16} 
			\multicolumn{1}{l}{}        &                      & \multicolumn{1}{l}{iter} & \multicolumn{1}{l}{feval} & time($ms$)          & \multicolumn{1}{l}{stepsize} & \multicolumn{1}{c}{} & iter          & feval          & time($ms$)          & stepsize & \multicolumn{1}{c}{} & iter          & feval          & time($ms$)          & stepsize \\ \hline
			Imbalance1                  &                      & 33.27                    & 35.68                     & 4.48          & 0.98                         &                      & 19.92         & 66.13          & 4.06          & 1.00     &                      & \textbf{2.88} & \textbf{5.38}  & \textbf{0.56} & 0.75     \\
			Imbalance2                  &                      & 245.40                   & 1958.63                   & 72.68         & 0.01                         &                      & 245.95        & 1717.65        & 65.19         & 0.01     &                      & \textbf{1.49} & \textbf{5.36}  & \textbf{0.42} & 0.51     \\
			JOS1a                       &                      & 198.35                   & 198.35                    & 28.04         & 1.00                         &                      & \textbf{2.00} & \textbf{2.00}  & \textbf{0.48} & 13.00    &                      & \textbf{2.00} & \textbf{2.00}  & 0.51          & 1.00     \\
			JOS1b                       &                      & 383.44                   & 383.44                    & 68.55         & 1.00                         &                      & \textbf{2.00} & \textbf{2.00}  & 0.51          & 25.50    &                      & \textbf{2.00} & \textbf{2.00}  & \textbf{0.50} & 1.00     \\
			JOS1c                       &                      & 500.00                   & 500.00                    & 77.42         & 1.00                         &                      & \textbf{2.00} & \textbf{2.00}  & \textbf{0.52} & 25.50    &                      & \textbf{2.00} & \textbf{2.00}  & 0.54          & 1.00     \\
			JOS1d                       &                      & 500.00                   & 500.00                    & 83.26         & 1.00                         &                      & \textbf{2.00} & \textbf{2.00}  & \textbf{0.51} & 25.50    &                      & \textbf{2.00} & \textbf{2.00}  & 0.52          & 1.00     \\
			WIT1                        &                      & 35.27                    & 189.07                    & 8.26          & 0.15                         &                      & 30.73         & 145.85         & 7.74          & 0.07     &                      & \textbf{2.55} & \textbf{4.72}  & \textbf{0.63} & 0.72     \\
			WIT2                        &                      & 75.98                    & 459.34                    & 18.58         & 0.05                         &                      & 71.37         & 345.89         & 16.95         & 0.03     &                      & \textbf{3.32} & \textbf{6.60}  & \textbf{0.69} & 0.70     \\
			WIT3                        &                      & 31.43                    & 100.50                    & 5.61          & 0.35                         &                      & 19.52         & 58.57          & 4.08          & 0.12     &                      & \textbf{3.67} & \textbf{6.08}  & \textbf{0.86} & 0.74     \\
			WIT4                        &                      & 35.33                    & 47.60                     & 4.74          & 0.84                         &                      & 3.54          & 4.79           & \textbf{0.58} & 0.36     &                      & \textbf{3.33} & \textbf{4.58}  & 0.66          & 0.80     \\
			WIT5                        &                      & 37.68                    & 49.49                     & 5.49          & 0.85                         &                      & 3.33          & 4.37           & \textbf{0.63} & 0.40     &                      & \textbf{3.22} & \textbf{4.33}  & 0.70          & 0.81     \\
			WIT6                        &                      & \textbf{1.00}            & \textbf{2.00}             & \textbf{0.33} & 0.50                         &                      & \textbf{1.00} & \textbf{2.00}  & 0.37          & 0.50     &                      & \textbf{1.00} & \textbf{2.00}  & 0.36          & 0.50     \\
			Deb                         &                      & 71.03                    & 431.83                    & 20.61         & 0.09                         &                      & 64.21         & 374.68         & 20.14         & 0.04     &                      & \textbf{3.95} & \textbf{7.03}  & \textbf{0.97} & 0.74     \\
			PNR                         &                      & 10.72                    & 28.59                     & 1.77          & 0.45                         &                      & 2.86          & 5.23           & \textbf{0.60} & 0.24     &                      & \textbf{2.69} & \textbf{4.47}  & 0.62          & 0.74     \\
			DD                          & \multicolumn{1}{c}{} & 37.68                    & 38.91                     & 5.19          & 0.98                         &                      & 34.26         & 123.07         & 7.57          & 1.03     &                      & \textbf{7.44} & \textbf{7.99}  & \textbf{1.68} & 0.96     \\
			FDS                         & \multicolumn{1}{c}{} & 51.06                    & 69.56                     & 38.39         & 0.95                         &                      & 39.48         & 182.72         & 52.59         & 1.09     &                      & \textbf{7.34} & \textbf{7.39}  & \textbf{5.54} & 0.99     \\
			TRIDIA\_1                   & \multicolumn{1}{c}{} & 25.83                    & 58.19                     & 10.65         & 0.59                         &                      & 15.08         & 19.28          & \textbf{6.16} & 0.85     &                      & \textbf{8.90} & \textbf{11.78} & 6.97          & 0.76     \\
			TRIDIA\_2                   & \multicolumn{1}{c}{} & 36.34                    & 63.62                     & 33.01         & 0.70                         &                      & \textbf{9.44} & \textbf{11.19} & \textbf{7.39} & 0.84     &                      & 12.14         & 18.46          & 31.01         & 0.78     \\ \hline
		\end{tabular}
	}
	\caption{Number of average iterations (iter), number of average function evaluations (feval), average CPU time (time($ms$)) and average stepsize (stepsize) of SDMO, BBMO, and BBDMO implemented on different test problems with \textbf{average-type nonmonotone} line search.}
	\label{tab4}
\end{table}

For each test problem, the number of average iterations (iter), number of average function evaluations (feval), average CPU time (time($ms$)), and average stepsize (stepsize) of the different algorithms with monotone and nonmonotone line search are listed in Tables \ref{tab2}-\ref{tab4}, respectively. The numerical results confirmed that BBDMO outperforms SDMO and BBMO in terms of average iterations, average function evaluations, and average CPU time. The average stepsize of BBDMO is robust and in $[0.5,1]$ for different problems and line search techniques. While the average stepsize of SDMO and BBMO varies significantly in different problems and is enlarged by nonmonotone line search techniques except for problems JOS1a-d. Comparing Table \ref{tab2} with Tables \ref{tab3} and \ref{tab4}, we can see the improvement of average iterations and average function evaluations required by SDMO and BBMO is limited with nonmonotone line search techniques, which shows that nonmonotone line search techniques can not address the drawback of SDMO effectively. BBMO provides an appropriate initial stepsize along with the steepest descent direction, which improves the performance for problems JOS1a-d and TRIDIA2 due to quadratic objective functions. However, there is little improvement for other test problems. Furthermore, SDMO and BBMO perform poorly on problems Imbalance2, WIT1-2, Deb and FDS, which have imbalanced objective functions, such as higher-order and exponential functions. Hence, the weakness of SDMO is due to the descent direction (the descent direction of BBMO is the same as SDMO), which does not consider the imbalances among objective functions. In view of the performance of BBDMO on problems Imbalance2, WIT1-2, Deb and FDS, we conclude that BBDMO can handle the imbalances effectively.
\section{Conclusions}
In this paper, we have pointed out that the small stepsize is mainly due to imbalances among objective functions in SDMO. Then, we provided a new method to address the issue, called BBDMO. Over the past decades, multitask learning has received plenty of attention in the machine learning community. More and more researchers cast multitask learning as multiobjective optimization, and then effective algorithms \citep[see e.g.][]{SK2018,LZ2019,MR2020} are devised based on SDMO to train the model. However, \citet{CB2018} pointed out that: \textit{``Task
	imbalances impede proper training because they manifest as imbalances between backpropagated gradients."} Fortunately, BBDMO is a first-order method that can handle the imbalances among objective functions effectively and thus a potential method for multitask learning as multiobjective optimization. It is worth noting that imbalanced objective functions are commonplace in large-scale MOPs. Hence, it is meaningful to apply BBDMO to solve large-scale MOPs, which is left as our future work.\\

\biboptions{authoryear}
\bibliography{mybibfile}
\end{document}